\documentclass[10pt,reqno]{amsart}
\usepackage{a4wide}
\usepackage{amssymb}
\usepackage{mathrsfs}
\usepackage{dsfont} 	
\usepackage[utf8]{inputenc}
\usepackage{t1enc}
\usepackage{mathabx}  
\usepackage{graphicx} 
\usepackage{fge} 

\newtheorem{proposition}{Proposition}
  \newtheorem{theorem}[proposition]{Theorem}
  \newtheorem{lemma}[proposition]{Lemma}
  \newtheorem{corollary}[proposition]{Corollary}
\theoremstyle{definition}
  \newtheorem{definition}[proposition]{Definition}
  \newtheorem{remark}[proposition]{Remark}
  
\numberwithin{equation}{section}
\numberwithin{proposition}{section}

\DeclareMathOperator{\Mor}{Mor}
\DeclareMathOperator{\M}{\mathsf{M}}
\DeclareMathOperator{\C}{C}
\DeclareMathOperator{\B}{B}

\newcommand{\cst}{\ifmmode\mathrm{C}^*\else{$\mathrm{C}^*$}\fi}
\newcommand{\tens}{\otimes}
\newcommand{\id}{\mathrm{id}}
\newcommand{\comp}{\circ}
\newcommand{\I}{\mathds{1}}
\newcommand{\bz}{\boldsymbol{z}}
\newcommand{\eps}{\varepsilon}
\newcommand{\mysetminus}{{\;\!\scriptstyle{\fgebackslash}\;\!}}

\newcommand{\CC}{\mathbb{C}}
\newcommand{\TT}{\mathbb{T}}
\newcommand{\ZZ}{\mathbb{Z}}
\renewcommand{\SS}{\mathbb{S}}
\newcommand{\VV}{\mathbb{V}}
\newcommand{\NN}{\mathbb{N}}
\newcommand{\RR}{\mathbb{R}}

\newcommand{\sA}{\mathsf{A}}
\newcommand{\sB}{\mathsf{B}}

\newcommand{\cH}{\mathscr{H}}

\newcommand{\btens}{\boxtimes}
\newcommand{\SUq}{\operatorname{SU}_q(2)}
\newcommand{\CstT}{\mathfrak{C}^*_\mathbb{T}}
\newcommand{\XX}{\mathbb{X}}
\newcommand{\bn}{\boldsymbol{n}}
\newcommand{\bu}{\boldsymbol{u}}
\newcommand{\bV}{\boldsymbol{V}}
\newcommand{\bv}{\boldsymbol{v}}
\newcommand{\DSUq}{\Delta_{\scriptscriptstyle{\operatorname{SU}_q\!(2)}\!}}

\begin{document}

\title[Braided Quantum Spheres]{Podle\'s spheres for the Braided Quantum $\operatorname{SU}(2)$}

\author[P.M.~So{\l}tan]{Piotr M.~So{\l}tan}
\address{Department of Mathematical Methods in Physics, Faculty of Physics, University of Warsaw, Poland}
\email{piotr.soltan@fuw.edu.pl}

\begin{abstract}
Starting with the braided quantum group $\SUq$ for a complex deformation parameter $q$ we perform the construction of the quotient $\operatorname{SU}_q(2)/\TT$ which serves as a model of a quantum sphere. Then we follow the reasoning of Podle\'{s} who for real $q$ classified quantum spaces with the action of $\operatorname{SU}_q(2)$ with appropriate spectral properties. These properties can also be expressed in the context of the braided quantum $\operatorname{SU}_q(2)$ (with complex $q$) and we find that they lead to precisely the same family of quantum spaces as found by Podle\'{s} for the real parameter $|q|$.
\end{abstract}



\maketitle



\section{Introduction}

Quantum spheres were defined and analyzed by Podle\'s in his \cite{spheres} and later in \cite{podlesdiff,podlesdiff2,podles} (see also \cite{podlesqm}). These remarkable quantum spaces were the source of tremendous activity in the field of non-commutative geometry (see \cite{ldgeometry} and references therein). They were defined as homogeneous spaces for the quantum group $\SUq$ defined in \cite{twisted} satisfying conditions which for $q=1$ characterize the two-dimensional sphere $\SS^2$. The quantum group and consequently the quantum spheres were defined for a real deformation parameter $q\in[-1,1]\mysetminus\{0\}$ (in fact, as it is pointed out in \cite[Remark 3]{spheres} the quantum spheres do not depend on the sign of $q$).

In a more recent development a deformation of the \cst-algebra $\C(\operatorname{SU}(2))$ of continuous functions on the compact group $\operatorname{SU}(2)$ for a complex $q$ satisfying $0<|q|<1$ was considered in \cite{KMRW} (but see also \cite{dhhmw}). This new deformation no longer carries the structure of a compact quantum group, but is instead a \emph{compact braided quantum group}. This means that for complex $q$ the algebra $\C(\SUq)$ has a coproduct whose values lie in a certain \emph{braided} product of $\C(\SUq)$ with itself instead of the usual minimal tensor product. Braided quantum groups were introduced studied by Majid e.g.~in \cite{majid91braidedgr,majid93braidedlin,majid93braidedgroups,majidfoundations} in the purely algebraic context. The \cst-algebraic version of this theory was first studied in \cite{Karpacz} and later in \cite{KMRW} and \cite{twisted2}. See also \cite{proj}, where braided quantum groups arise naturally in the von Neumann algebraic description of semidirect products of locally compact quantum groups. The connection with semidirect products was already hinted at in \cite{wysoczanski} (for real $q$) and developed fully in \cite[Section 6]{KMRW}.

In this paper we study the construction of Podle\'{s} from \cite{spheres} for the braided quantum group $\SUq$. However we will not be developing the theory of actions and isotypical components used by Podle\'{s}, but instead focus on the analogy with the quantum sphere obtained by a construction fully analogous to the quotient of a (non-braided) quantum  group by a closed quantum subgroup. In the work of Podle\'{s} this construction leads to the so called \emph{standard sphere} and the other quantum spheres share many properties of the standard sphere. We will repeat the quotient construction in the case of complex $q$, i.e.~for the \emph{braided} quantum group $\SUq$. Later we will use the notion of a tensor product of representations of a braided quantum group from \cite{KMRW} to define an irreducible three-dimensional representation of $\SUq$ and find all compact quantum spaces with an action of $\SUq$ (in a braided sense, see Sections \ref{theQuotientS2}, \ref{simplespaces} and \ref{bps}) which resembles the action on the sphere obtained via the quotient construction. It turns out that these quantum spaces are exactly the quantum spaces found by Podle\'{s} for the real deformation parameter $|q|$. Let us point out that Podle\'{s} worked with the \emph{left} quotient $\TT\,\backslash\SUq$ and hence got a \emph{right} action of $\SUq$ on the quotient sphere, while we chose to use the \emph{right} quotient $\SUq/\TT$ and therefore obtained a \emph{left} action of $\SUq$ on the quotient sphere. This is, of course, of little importance, but this convention works better in the context in which all considered \cst-algebras come equipped with a left action of $\TT$ which is necessary for the braided monoidal structure leading to the braided $\SUq$ (cf.~\cite{KMRW} and Section \ref{braids} and \ref{bSUq}).

The quantum spaces underlying the braided quantum groups $\SUq$ (for all complex $q$ such that $|q|<1$) are the same, i.e.~the corresponding \cst-algebras are isomorphic. This was proved already in \cite[Appendix A.2]{twisted} for real $q$, then stated in \cite[Section 3.1]{dhhmw} and proved in \cite[Section 2]{KMRW}. Similarly Podle\'{s} spheres are all isomorphic to a certain extension $\mathcal{K}\oplus\mathcal{K}$ by $\C(\TT)$, where $\mathcal{K}$ is the algebra of compact operators on a separable Hilbert space (\cite{sheu}, see also \cite[Proposition 4]{spheres}) or to the minimal unitization of $\mathcal{K}$ (the standard quantum sphere). Incidentally all these algebras can be viewed as graph \cst-algebras of certain directed graphs (\cite{HoSz}). In particular, one can interpret the braided quantum group structure of $\SUq$ for complex $q$ as yet another structure on the same quantum space which, as a quantum space, is independent of $q$.

We will now briefly describe the contents of the paper: in Section \ref{braids} we gather certain facts about braided tensor products needed for the construction of the braided $\SUq$ and the corresponding spheres. Section \ref{bSUq} recalls the details of the braided version of $\operatorname{SU}(2)$ and in Section \ref{theQuotientS2} we study the quotient sphere
\[
\SS^2_q=\SUq/\TT.
\]
Section \ref{theV} is devoted to the definition and study of the three-dimensional irreducible representation of $\SUq$ found as an irreducible subrepresentation of the tensor square of the fundamental representation (cf.~\cite[Section 5]{KMRW}) and in Section \ref{simplespaces} we find the commutation relations implied by existence of an action of the braided quantum group $\SUq$ satisfying certain conditions suggested by the quotient construction from Section \ref{theQuotientS2} (cf. Section \ref{inTerms}). For real $q$ these conditions were interpreted by Podle\'{s} as saying that each irreducible representation of the quantum group $\operatorname{SO}_q(3)$ appears with multiplicity $1$ in the spectrum of the action. While such an interpretation is possible also in the braided context, we choose to simply follow the example of the quotient sphere. In Section \ref{bps} we show that \cst-algebras defined by relations derived in Section \ref{simplespaces} indeed carry and action of the braided $\SUq$ and that they coincide with Podle\'{s} spheres from \cite{spheres}.

The appendices \ref{vv} and \ref{directcheck} contain certain simple computations which are useful in other sections.

Throughout the paper $q$ is a complex parameter such that $0<|q|<1$ and we set
\[
\zeta=\tfrac{q}{\overline{q}}.
\]
In Sections \ref{theV}, \ref{simplespaces}, \ref{bps} and the Appendix we will use the following shorthands:
\[
\sigma=\sqrt{1+|q|^2},\qquad\varsigma=|q|^2
\]
which will lighten the notation. Furthermore we will use the following conventions:
\begin{itemize}
\item if $\sA$ is a normed algebra and $A,B\subset\sA$ then the symbol $[AB]$ will denote the closure of $\operatorname{span}\{ab\left.\right|\,a\in{A},\:b\in{B}\}$,
\item given $n\in\NN$ we will denote by $M_n$ the \cst-algebra of $n\times{n}$ matrices over $\CC$ which will also be identified with $\B(\CC^n)$.
\end{itemize}

\section{The category of $\TT$-\cst-algebras and its monoidal structure}\label{braids}

In this section we summarize the elements of \cite{KMRW,MRWtwist} which we need to state our results. We will also try to stick to the notation and terminology introduced in those papers.

For \cst-algebras $\sA$ and $\sB$ the symbol $\Mor(\sA,\sB)$ will denote the set of \emph{morphisms} between \cst-algebras as defined in \cite[Section 0]{unbo}. Since all \cst-algebras considered in this paper will be unital, all morphisms will simply be unital $*$-homomorphism.

Following \cite{KMRW} (cf.~also \cite{MRWtwist,Karpacz}) we denote by $\CstT$ the category of \cst-algebras equipped with an action of the circle group $\TT$ and equivariant morphisms. For each $\sA\in\operatorname{Ob}(\CstT)$ the corresponding action of $\TT$ is described by $\rho^\sA\in\Mor(\sA,\C(\TT)\tens\sA)$. For $\sA,\sB\in\operatorname{Ob}(\CstT)$ we will use $\Mor_\TT(\sA,\sB)$ to denote the set of morphisms from $\sA$ to $\sB$ in the category $\CstT$, i.e.~those $\Phi\in\Mor(\sA,\sB)$ which satisfy
\[
\rho^\sB\comp\Phi=(\id\tens\Phi)\comp\rho^\sA.
\]
Often we will use an alternative picture of actions of $\TT$, namely for $\sA\in\operatorname{Ob}(\CstT)$ instead of $\rho^{\sA}$ we can introduce automorphisms $\bigl\{\rho^{\sA}_z\,\bigl.\bigr|\,z\in\TT\bigr\}$ of $\sA$ by setting
\[
\rho^{\sA}_z=(\delta_z\tens\id)\comp\rho^{\sA},\qquad{z}\in\TT,
\]
where $\delta_z$ is the evaluation functional $\C(\TT)\ni{f}\mapsto{f(z)}\in\CC$. The canonical generator of $\C(\TT)$, i.e.~the identity function
\[
\TT\ni{z}\longmapsto{z}\in\CC
\]
will be denoted by $\bz$.

We will now recall the monoidal structure on the category $\CstT$ necessary for the construction of the braided quantum $\operatorname{SU}(2)$ group. For this we will need the \cst-algebra of functions on the quantum two-torus (\cite{connes1980,rieffelTori}). Given a complex number $\mu$ of absolute value $1$ the \cst-algebra $\C(\TT^2_\mu)$ of continuous functions on the quantum two-torus for the deformation parameter $\mu$ is defined as the universal \cst-algebra generated by two unitary elements $u$ and $v$ such that $uv=\mu{v}u$.

The monoidal structure on $\CstT$ is a bifunctor $\btens_\zeta$ from $\CstT\times\CstT$ to $\CstT$ which on objects is given by
\[
\sA\btens_\zeta\sB=\bigl[j_1(\sA)j_2(\sB)\bigr],
\]
where $j_1:\sA\hookrightarrow\M(\C(\TT^2_\zeta)\tens\sA\tens\sB)$ and $j_2:\sB\hookrightarrow\M(\C(\TT^2_\zeta)\tens\sA\tens\sB)$ are defined by
\[
j_1(a)=u^n\tens{a}\tens\I,\qquad{j_2(b)}=v^m\tens\I\tens{b}
\]
for $a\in\sA$ such that $\rho^\sA(a)=\bz^n\tens{a}$ and $b\in\sB$ such that $\rho^\sB(b)=\bz^m\tens{b}$ (in what follows we will say that $a$ and $b$ are \emph{homogeneous} of degree $n$ and $m$ respectively, and write $\deg(a)=n$ and $\deg(b)=m$). The \cst-algebra $\sA\btens_\zeta\sB$ is called the \emph{crossed product} in \cite{MRWtwist} (cf.~also \cite{Karpacz}) and a \emph{braided tensor product} in \cite{KMRW}. We will stick with the latter terminology as it coincides with the one used earlier by Majid e.g.~in \cite{majid93braidedlin,majid91braidedgr,majidfoundations}.

The morphisms $j_1$ and $j_2$ viewed as maps from $\sA$ and $\sB$ to $\sA\btens_\zeta\sB$ will from now on be denoted by $\jmath_1$ and $\jmath_2$ and will be understood as part of the structure of $\sA\btens_\zeta\sB$. In particular we will use the same symbols ($\jmath_1$ and $\jmath_2$) for any $\TT$-\cst-algebras involved. There is a unique action of $\TT$ on $\sA\btens_\zeta\sB$ such that both $\jmath_1$ and $\jmath_2$ are equivariant: $\jmath_1\in\Mor_\TT(\sA,\sA\btens_\zeta\sB)$ and $\jmath_2\in\Mor_\TT(\sB,\sA\btens_\zeta\sB)$.

Occasionally we will use $\jmath_1,\jmath_2$ and $\jmath_3$ for the canonical inclusions of factors in a triple braided product.

\begin{remark}\label{actionT2}
The construction of $\sA\btens_\zeta\sB$ shows that apart from the action of $\TT$ on this \cst-algebra there is also an action of $\TT^2$ which comes from the action of $\TT^2$ on $\TT^2_\zeta$ by scaling $u$ and $v$ by two different complex numbers of modulus $1$. The canonical action of $\TT$ on $\sA\btens_\zeta\sB$ is then the restriction of the action of $\TT^2$ to the diagonal subgroup.
\end{remark}

The action of $\btens_\zeta$ on morphisms assigns to $\Phi\in\Mor_\TT(\sA,\sA')$ and $\Psi\in\Mor_\TT(\sB,\sB')$ the unique element of $\Phi\btens_\zeta\Psi\in\Mor_\TT(\sA\tens\sB,\sA'\tens\sB')$ such that
\[
(\Phi\btens_\zeta\Psi)\bigl(\jmath_1(a)\jmath_2(b)\bigr)=
\jmath_1\bigl(\Phi(a)\bigr)\jmath_2\bigl(\Psi(b)\bigr),\qquad{a}\in\sA,\:b\in\sB
\]
(\cite{MRWtwist}, \cite[Section 3]{KMRW}).

\begin{remark}
In some calculations the following formula is of great importance: for $\sA,\sB\in\operatorname{Ob}(\CstT)$ and homogeneous elements $a\in\sA$, $b\in\sB$ we have
\begin{equation}\label{braid}
\jmath_2(a)\jmath_1(a)=\overline{\zeta}^{\,\deg(a)\deg(b)}\jmath_1(b)\jmath_2(a)
\end{equation}
(\cite[Formula (1.5)]{KMRW}).
\end{remark}

\begin{remark}\label{subalg}
If $\mathscr{A}\subset\sA$ and $\mathscr{B}\subset\sB$ are $\TT$-invariant \cst-subalgebras then $\mathscr{A}\btens_\zeta\mathscr{B}$ is a subalgebra of $\sA\btens_\zeta\sB$ in the natural way, i.e.~the embeddings $\jmath_1$ and $\jmath_2$ of $\mathscr{A}$ and $\mathscr{B}$ into $\mathscr{A}\btens_\zeta\mathscr{B}$ are restrictions of $\jmath_1\in\Mor_\TT(\sA,\sA\btens_\zeta\sB)$ and $\jmath_2\in\Mor_\TT(\sB,,\sA\btens_\zeta\sB)$ to $\mathscr{A}$ and $\mathscr{B}$. This follows easily from the properties of the minimal tensor product (see e.g.~\cite[Section 1.2]{wassermann}).
\end{remark}

Of course \cst-algebras with trivial action of $\TT$ are also objects of $\CstT$. It can be checked that if $\sA,\sB\in\operatorname{Ob}(\CstT)$ and $\rho^\sB(b)=\I\tens{b}$ for all $b\in\sB$ then $\sA\btens_\zeta\sB\cong\sA\tens\sB$ with $\jmath_1$ and $\jmath_2$ the usual embeddings onto the first and second leg of the tensor product (\cite[Example 3.13]{MRWtwist}).

\section{The braided quantum $\operatorname{SU}(2)$ group}\label{bSUq}

Following \cite{KMRW} we define $\C(\SUq)$ to be the universal \cst-algebra generated by two elements $\alpha$ and $\gamma$ with relations
\begin{equation}\label{SUq2rel}
\begin{aligned}
\alpha^*\alpha+\gamma^*\gamma&=\I,&\alpha\gamma&=\overline{q}\gamma\alpha,\\
\alpha\alpha^*+|q|^2\gamma^*\gamma&=\I,&\gamma\gamma^*&=\gamma^*\gamma.
\end{aligned}
\end{equation}

\begin{remark}
Relations \eqref{SUq2rel} imply that
\begin{equation}\label{add}
\alpha\gamma^*=q\gamma^*\alpha.
\end{equation}
Indeed, since $\alpha^*\alpha=\I-\gamma^*\gamma$ and $\gamma$ is normal, we have
\[
\gamma\alpha^*\alpha\gamma^*
=\gamma(\I-\gamma^*\gamma)\gamma^*
=\gamma\gamma^*-(\gamma\gamma^*)^2
=(\I-\gamma\gamma^*)\gamma\gamma^*
=(\I-\gamma^*\gamma)\gamma\gamma^*
=\alpha^*\alpha\gamma\gamma^*.
\]
Similarly $|q|^2\gamma\gamma^*=\I-\alpha\alpha^*$, so
\[
|q|^2\alpha^*\gamma\gamma^*\alpha=\alpha^*(\I-\alpha\alpha^*)\alpha=\alpha^*\alpha-(\alpha^*\alpha)^2=(\I-\alpha^*\alpha)\alpha^*\alpha=\gamma^*\gamma\alpha^*\alpha=\gamma\gamma^*\alpha^*\alpha.
\]
Using these identities and the fact that $\alpha\gamma=\overline{q}\gamma\alpha$ we immediately find
{\allowdisplaybreaks
\begin{align*}
(\alpha\gamma^*-q\gamma^*\alpha)^*(\alpha\gamma^*-q\gamma^*\alpha)
&=\gamma\alpha^*\alpha\gamma^*-\overline{q}\alpha^*\gamma\alpha\gamma^*
-q\gamma\alpha^*\gamma^*\alpha+|q|^2\alpha^*\gamma\gamma^*\alpha\\
&=\gamma\alpha^*\alpha\gamma^*-\alpha^*\alpha\gamma\gamma^*
-q\gamma\alpha^*\gamma^*\alpha+|q|^2\alpha^*\gamma\gamma^*\alpha\\
&=-q\gamma\alpha^*\gamma^*\alpha+|q|^2\alpha^*\gamma\gamma^*\alpha\\
&=-\gamma\gamma^*\alpha^*\alpha+|q|^2\alpha^*\gamma\gamma^*\alpha\\
&=-\gamma\gamma^*\alpha^*\alpha+\gamma\gamma^*\alpha^*\alpha=0.
\end{align*}
} 
\end{remark}

The following lemma is a simple consequence of the defining relations of $\C(\SUq)$ (cf.~\cite[Theorem 1.2]{twisted}):

\begin{lemma}\label{baza}
For $n\in\ZZ$ and $k,l\in\ZZ_+$ define
\[
a_{n,k,l}=\begin{cases}
\alpha^n\gamma^k{\gamma^*}^l&n\geq{0},\\
{\alpha^*}^{-n}\gamma^k{\gamma^*}^l&n<0.
\end{cases}
\]
Then $\mathscr{A}=\operatorname{span}\bigl\{a_{n,k,l}\,\bigl.\bigr|\,n\in\ZZ,\:k,l\in\ZZ_+\bigr\}$ is a dense unital $*$-subalgebra of $\C(\SUq)$.
\end{lemma}

\begin{remark}
It can be shown that the system $\{a_{n,k,l}\}_{n\in\ZZ,k,l\in\ZZ_+}$ is linearly independednt, but we will not need this fact except for its much simpler consequence which we prove directly in Section \ref{irr}. Still it is very convenient for certain computations to express elements of $\mathscr{A}$ as linear combinations of this system.
\end{remark}

The \cst-algebra $\C(\SUq)$ is equipped with and action of $\TT$ described by
\[
\rho^{\C(\SUq)}\in\Mor\bigl(\C(\SUq),\C(\TT)\tens\C(\SUq)\bigr)
\]
determined by
\[
\rho^{\C(\SUq)}(\alpha)=\I\tens\alpha,\qquad\rho^{\C(\SUq)}(\gamma)=\bz\tens\gamma.
\]
Thus $\C(\SUq)$ becomes an object of the monoidal category $\CstT$.

We know from \cite[Theorem 1.1]{KMRW} that there exists a unique
\[
\DSUq\in\Mor_\TT\bigl(\C(\SUq),\C(\SUq)\btens_\zeta\C(\SUq)\bigr)
\]
such that
\begin{equation}\label{DeltaSUq}
\begin{split}
\DSUq(\alpha)=\jmath_1(\alpha)\jmath_2(\alpha)-q\jmath_1(\gamma^*)\jmath_2(\gamma),\\
\DSUq(\gamma)=\jmath_1(\gamma)\jmath_2(\alpha)+\jmath_1(\alpha^*)\jmath_2(\gamma).
\end{split}
\end{equation}
The morphism $\DSUq$ defines the structure of a braided quantum group on $\SUq$, i.e.~it is coassociative:
\[
(\DSUq\btens_\zeta\id)\comp\DSUq=(\id\btens_\zeta\DSUq)\comp\DSUq
\]
and we have
\begin{equation}\label{density}
\begin{aligned}
&\bigl[\DSUq\bigl(\C(\SUq)\bigr)\jmath_2\bigl(\C(\SUq)\bigr)\bigr]\\
&\quad=\bigl[\jmath_1\bigl(\C(\SUq)\bigr)\DSUq\bigl(\C(\SUq)\bigr)\bigr]
=\C(\SUq)\btens_\zeta\C(\SUq).
\end{aligned}
\end{equation}

\section{The quotient sphere}\label{theQuotientS2}

Consider the \cst-algebra $\C(\TT)$ with trivial action of $\TT$. Then $\C(\TT)$ is an object of $\CstT$ and for any other $\sA\in\operatorname{Ob}(\CstT)$ we have $\sA\btens_\zeta\C(\TT)=\sA\tens\C(\TT)$ with $\jmath_1(a)=a\tens\I$ for all $a\in\sA$ and $\jmath_2(f)=\I\tens{f}$ for all $f\in\C(\TT)$ (cf.~Section \ref{braids}). Moreover it is easy to see that we have $\Delta_\TT\in\Mor_\TT(\C(\TT),\C(\TT)\btens_\zeta\C(\TT))$ and $\Delta_\TT(\bz)=\jmath_1(\bz)\jmath_2(\bz)$ or in other words $\Delta_\TT(\bz)=\bz\tens\bz$).

Consider further $\pi\in\Mor_\TT(\C(\SUq),\C(\TT))$ defined by
\[
\pi(\alpha)=\bz,\qquad\pi(\gamma)=0.
\]
\begin{proposition}
We have
\begin{equation}\label{equivar}
(\pi\btens_\zeta\pi)\comp\DSUq=\Delta_\TT\comp\pi.
\end{equation}
\end{proposition}

\begin{proof}
It is a simple computation to check \eqref{equivar} on generators of $\C(\SUq)$. Indeed:
\begin{align*}
(\pi\btens_\zeta\pi)\DSUq(\alpha)
&=(\pi\btens_\zeta\pi)\bigl(\jmath_1(\alpha)\jmath_2(\alpha)-q\jmath_1(\gamma^*)\jmath_2(\gamma)\bigr)\\
&=\jmath_1\bigl(\pi(\alpha)\bigr)\jmath_2\bigl(\pi(\alpha)\bigr)-q\jmath_1\bigl(\pi(\gamma^*)\bigr))\jmath_2\bigl(\pi(\gamma)\bigr)\\
&=\jmath_1(\bz)\jmath_2(\bz)=\Delta_\TT(\bz)=\Delta_\TT\bigl(\pi(\alpha)\bigr)
\end{align*}
and similarly
\begin{align*}
(\pi\btens_\zeta\pi)\DSUq(\alpha)
&=(\pi\btens_\zeta\pi)
\bigl(\jmath_1(\gamma)\jmath_2(\alpha)+\jmath_1(\alpha^*)\jmath_2(\gamma)\bigr)\\
&=\jmath_1\bigl(\pi(\gamma)\bigr)\jmath_2\bigl(\pi(\alpha)\bigr)+\jmath_1\bigl(\pi(\alpha^*)\bigr)\jmath_2\bigl(\pi(\gamma)\bigr)\\
&=0=\Delta_\TT\bigl(\pi(\gamma)\bigr).
\end{align*}
\end{proof}

It follows that $\TT$ is a (closed braided quantum) subgroup of $\SUq$. Following the standard procedure (\cite{spheres,podles,wangPhD}) let us define the algebra of functions on the homogeneous space $\SUq/\TT$:
\[
\C(\SS^2_q)=\bigl\{a\in\C(\SUq)\,\bigl.\bigr|\,(\id\btens_\zeta\pi)\DSUq(a)=a\tens\I\bigr\}
\]
(note that we can write the condition $(\id\btens_\zeta\pi)\DSUq(a)=a\tens\I$ because $(\id\btens_\zeta\pi)\DSUq(a)$ belongs to $\C(\SUq)\btens_\zeta\C(\TT)=\C(\SUq)\tens\C(\TT)$).

\begin{proposition}\label{sigma}
The map $\sigma=(\id\btens_\zeta\pi)\comp\DSUq\in\Mor_\TT(\C(\SUq),\C(\SUq)\tens\C(\TT))$ is a right action of $\TT$ on $\SUq$.
\end{proposition}

\begin{proof}
On generators of $\C(\SUq)$ we have
\begin{align*}
(\id\btens_\zeta\pi)\DSUq(\alpha)
&=(\id\btens_\zeta\pi)\bigl(\jmath_1(\alpha)\jmath_2(\alpha)-q\jmath_1(\gamma^*)\jmath_2(\gamma)\bigr)=\jmath_1(\alpha)\jmath_2(\bz)
=\alpha\tens\bz,\\
(\id\btens_\zeta\pi)\DSUq(\gamma)
&=(\id\btens_\zeta\pi)\bigl(\jmath_1(\gamma)\jmath_2(\alpha)+\jmath_1(\alpha^*)\jmath_2(\gamma)\bigr)=\jmath_1(\gamma)\jmath_2(\bz)=\gamma\tens\bz
\end{align*}
and from this it follows immediately that $(\sigma\tens\id)\comp\sigma=(\id\tens\Delta_\TT)\comp\sigma$.
\end{proof}

As a consequence $\C(\SS^2_q)$ can be identified with the fixed point subalgebra $\C(\SUq)^\sigma$ of $\C(\SUq)$ for the right action $\sigma$. Note that since $\sigma$ is equivariant (for the left actions of $\TT$), the subalgebra $\C(\SS^2_q)\subset\C(\SUq)$ is invariant under the action $\rho^{\C(\SUq)}$. In particular it is an object of $\CstT$ with $\rho^{\C(\SS^2_q)}$ equal to the restriction of $\rho^{\C(\SUq)}$ to $\C(\SS^2_q)$.

\begin{corollary}\label{genXq}
$\C(\SS^2_q)$ is the closed unital $*$-subalgebra of $\C(\SUq)$ generated by $\alpha\gamma^*$ and $\gamma^*\gamma$.
\end{corollary}

\begin{proof}
Denote by $\sB$ the closed unital $*$-subalgebra of $\C(\SUq)$ generated by $\alpha\gamma^*$ and $\gamma^*\gamma$. Since both of its generators are invariant for $\sigma$, we see that $\sB\subset\C(\SUq)^\sigma$.

Conversely, let us take any $x\in\C(\SUq)^\sigma$. As any element of $\C(\SUq)$ it can be written as a limit
\[
x=\lim_{n\to\infty}p_n
\]
of elements of $\mathscr{A}$ (cf.~Lemma \ref{baza}), i.e.~each $p_n$ is a linear combination elements of the system $\{a_{n,k,l}\}$. Now let $E:\C(\SUq)\to\C(\SUq)^\sigma$ be the projection onto fixed points given by integration with respect to the normalized Haar measure on $\TT$. Obviously we have
\[
x=E(x)=\lim_{n\to\infty}E(p_n).
\]
Moreover it is easy to see that
\[
E(a_{n,k,l})=\begin{cases}a_{n,k,l}&n=l-k\\0&n\neq{l-k}.\end{cases}
\]
In other words $x$ is a limit of linear combinations of elements of the set
\[
\bigl\{a_{l-k,k,l}\,\bigl.\bigr|\,l,k\in\ZZ_+\bigr\}.
\]
But each of these elements belongs to $\sB$, as each of the elements
\[
a_{l-k,k,l}=
\begin{cases}
q^{\frac{(l-k)(l-k-1)}{2}}(\alpha\gamma^*)^{l-k}(\gamma^*\gamma)^k&l\geq{k},\\
\overline{q}^{-\frac{(k-l)(k-l+1)}{2}}(\gamma\alpha^*)^{k-l}(\gamma^*\gamma)^l&l<k
\end{cases}
\]
belongs to $\sB$.
\end{proof}

Renaming the generators of $\C(\SS^2_q)$ as follows:
\[
A=\gamma^*\gamma,\qquad{B}=\alpha\gamma^*
\]
we easily find that
\begin{equation}\label{podlesRel}
\begin{aligned}
B^*B&=A^2-A^4,&BA&=|q|^2AB,\\
BB^*&=|q|^2A^2-|q|^4A^4,&A^*&=A.
\end{aligned}
\end{equation}

These are the relations defining the so called ``standard Podle\'{s} sphere'' (see \cite[Section 2.5]{ldgeometry}) for the deformation parameter $|q|$, i.e.~we have $\SS^2_q=S^2_{|q|}$ (in the notation of \cite[Section 6]{spheres}). Therefore
\begin{itemize}
\item $\C(\SS^2_q)$ defined as a \cst-subalgebra of $\C(\SUq)$ is in fact the universal unital \cst-algebra generated by elements $A$ and $B$ with relations \eqref{podlesRel},
\item $\C(\SS^2_q)$ is isomorphic to the minimal unitization of the compacts
\end{itemize}
(\cite{ldgeometry,spheres,podles}).

\begin{theorem}\label{quotientsphere}
Let $\Gamma=\bigl.\Delta\bigr|_{\C(\SS^2_q)}$. Then $\Gamma\in\Mor_\TT(\C(\SS^2_q),\C(\SUq)\btens_\zeta\C(\SS^2_q))$ and
\begin{enumerate}
\item\label{coass} $(\id\btens_\zeta\Gamma)\comp\Gamma=(\DSUq\btens_\zeta\id)\comp\Gamma$,
\item\label{podCond} $\bigl[\jmath_1\bigl(\C(\SUq)\bigr)\Gamma\bigl(\C(\SS^2_q)\bigr)\bigr]=\C(\SUq)\btens_\zeta\C(\SS^2_q)$,
\item for $x\in\C(\SS^2_q)$ we have $\Gamma(x)=\jmath_2(x)$ if and only if $x\in\CC\I$.
\end{enumerate}
\end{theorem}

\begin{proof}
The braided product $\C(\SUq)\btens_\zeta\C(\SS^2_q)$ can be naturally identified with the subalgebra of $\C(\SUq)\btens_\zeta\C(\SUq)$ generated by $\jmath_1(\C(\SUq))$ and $\jmath_2(\C(\SS^2_q))$ (cf.~Remark \ref{subalg}). Furthermore $\DSUq(\C(\SS^2_q))\subset\bigl[\jmath_1(\C(\SUq))\jmath_2(\C(\SS^2_q))\bigr]$ which can be checked either by using the fact that $(\DSUq\tens\id)\comp\sigma=(\id\btens_\zeta\sigma)\comp\DSUq$, or by checking values of $\Gamma$ on generators of $\C(\SS^2_q)$. Since all maps involved are $\TT$-equivariant unital $*$-homomorphisms of unital \cst-algebras, this proves that $\Gamma\in\Mor_\TT(\C(\SS^2_q),\C(\SUq)\btens_\zeta\C(\SS^2_q))$.

Statement \eqref{coass} follows immediately form coassociativity of $\DSUq$ and the fact that the range of the map $\DSUq$ restricted to $\C(\SS^2_q)$ lies in $\C(\SUq)\btens_\zeta\C(\SS^2_q)$.

The proof of Statement \eqref{podCond} follows the lines of analogous statements for non-braided quantum groups (see e.g.~\cite[Section 5]{exNonCpt}). Since braided products (in the sense of Section \ref{braids})  are not as familiar as ordinary tensor products, we will spell out all the details.

First let us recall from \cite[Section 5.2]{MRWtwist} that $\btens_\zeta$ is functorial not only for equivariant morphisms, but also for equivariant completely positive maps. Now since $\sigma$ (introduced in Proposition \ref{sigma}) is equivariant:
\[
(\id\tens\sigma)\comp\rho^{\C(\SUq)}=(\rho^{\C(\SUq)}\tens\id)\comp\sigma
\]
applying integration with respect to the normalized Haar measure on $\TT$ to the third leg we obtain
\[
(\id\tens{E})\comp\rho^{\C(\SUq)}=\rho^{\C(\SS^2_q)}\comp{E}
\]
where, as in the proof of Corollary \ref{genXq}, $E$ is the projection (in fact a conditional expectation) onto $\C(\SS^2_q)$.

It follows that we can form the map $\id\btens_\zeta{E}:\C(\SUq)\btens_\zeta\C(\SUq)\to\C(\SUq)\btens_\zeta\C(\SS^2_q)$ and it is easy to see that it is surjective. Thus, by the density conditions \eqref{density} we get
\begin{equation}\label{dwiestrony}
(\id\btens_\zeta{E})\bigl[\jmath_1\bigl(\C(\SUq)\bigr)\DSUq\bigl(\C(\SUq)\bigr)\bigr]
=\C(\SUq)\btens_\zeta\C(\SS^2_q).
\end{equation}
To see that the left hand of this equality is in fact $\bigl[\jmath_1(\C(\SUq))\Gamma(\C(\SS^2_q))\bigr]$ note first that
\[
\bigl[\jmath_1\bigl(\C(\SUq)\bigr)\DSUq\bigl(\C(\SUq)\bigr)\bigr]
\]
is equal to the closure of
\[
\operatorname{span}\bigl\{\jmath_1(a)\DSUq(b)\,\bigl.\bigr|\,a\in\C(\SUq),\:b\in\mathscr{A}\bigr\},
\]
where $\mathscr{A}$ is the dense subspace of $\C(\SUq)$ introduced in Lemma \ref{baza}. Now for any $b\in\mathscr{A}$ there are $b_1',\dotsc,b_N',b_2'',\dotsc,b_N''\in\mathscr{A}$ such that
\[
\DSUq(b)=\sum_{i}\jmath_1(b_i')\jmath_2(b_i''),
\]
so
{\allowdisplaybreaks
\begin{align*}
(\id\btens_\zeta{E})\bigl(\jmath_1(a)\DSUq(b)\bigr)
&=(\id\btens_\zeta{E})\biggl(\sum_{i}\jmath_1(ab_i')\jmath_2(b_i'')\biggr)\\
&=\sum_{i}\jmath_1(ab_i')\jmath_2\bigl(E(b_i'')\bigr)\\
&=\jmath_1(a)\biggl(\sum_{i}\jmath_1(b_i')\jmath_2\bigl(E(b_i'')\bigr)\biggr)\\
&=\jmath_1(a)(\id\btens_\zeta{E})\biggl(\sum_{i}\jmath_1(b_i')\jmath_2(b_i'')\biggr)\\
&=\jmath_1(a)(\id\btens_\zeta{E})\bigl(\DSUq(b)\bigr).
\end{align*}
} 

From coassociativity of $\DSUq$ we immediately infer that (denoting by $\int_\TT$ integration over $\TT$) we have
\begin{align*}
(\id\btens_\zeta{E})\comp\DSUq
&=\bigl(\id\btens_\zeta\id\btens_\zeta(\textstyle{\int_\TT}\comp\pi)\bigr)\comp(\id\btens_\zeta\DSUq)\comp\DSUq\\
&=\bigl(\id\btens_\zeta\id\btens_\zeta(\textstyle{\int_\TT}\comp\pi)\bigr)\comp(\DSUq\btens_\zeta\id)\comp\DSUq\\
&=\DSUq\comp\bigl(\id\btens_\zeta(\textstyle{\int_\TT}\comp\pi)\bigr)\comp\DSUq\\&=\DSUq\comp{E}=\Gamma\comp{E}.
\end{align*}
Therefore for $a\in\C(\SUq)$ and $b\in\mathscr{A}$
\[
(\id\btens_\zeta{E})\bigl(\jmath_1(a)\DSUq(b)\bigr)
=\jmath_1(a)(\id\btens_\zeta{E})\bigl(\DSUq(b)\bigr)
=\jmath_1(a)\Gamma\bigl(E(b)\bigr).
\]
Consequently the left hand side of \eqref{dwiestrony} is equal to the closure of
\[
\operatorname{span}\bigl\{\jmath_1(a)\Gamma\bigl(E(b)\bigr)\,\bigl.\bigr|\,a\in\C(\SUq),\:b\in\mathscr{A}\bigr\},
\]
which is $\bigl[\jmath_1\bigl(\C(\SUq)\bigr)\Gamma\bigl(\C(\SS^2_q)\bigr)\bigr]$.

Finally note that there exists a character $\eps$ of $\C(\SUq)$ mapping $\alpha$ to $1$ and $\gamma$ to $0$ (this is $\pi$ composed with evaluation in $1\in\TT$). Obviously $\eps\in\Mor_\TT(\C(\SUq),\CC)$ and
\[
(\id\btens_\zeta\eps)\comp\DSUq=(\eps\btens_\zeta\id)\comp\DSUq=\id
\]
(where we identify $\jmath_1:\C(\SUq)\to\C(\SUq)\btens_\zeta\CC$ as well as $\jmath_2:\C(\SUq)\to\CC\btens_\zeta\C(\SUq)$ with $\id:\C(\SUq)\to\C(\SUq)$). It follows that if $x\in\C(\SS^2_q)$ satisfies $\Gamma(x)=\jmath_2(x)$ then
\begin{align*}
x=(\id\btens_\zeta\eps)\bigl(\DSUq(x)\bigr)
&=(\id\btens_\zeta\eps)\bigl(\Gamma(x)\bigr)\\
&=(\id\btens_\zeta\eps)\bigl(\jmath_2(x)\bigr)\\
&=(\id\btens_\zeta\eps)\bigl(\jmath_1(\I)\jmath_2(x)\bigr)\\
&=\jmath_1(\I)\jmath_2\bigl(\eps(x)\bigr)=\eps(x)\I.
\end{align*}
\end{proof}

\section{The three-dimensional irreducible representation}\label{theV}

We begin by constructing the tensor square of the fundamental representation of the braided quantum $\operatorname{SU}(2)$ group. The first step is to choose an identification of $M_2\btens_\zeta{M_2}$ with $M_2\tens{M_2}$. Recall that the action of $\TT$ on $M_2$ we are considering is given by $\rho^{M_2}:M_2\to\C(\TT)\tens{M_2}$ determined uniquely by
\[
\rho^{M_2}(\bn)=\bz\tens\bn,
\]
where
\[
\bn=\begin{bmatrix}0&1\\0&0\end{bmatrix}
\]
is a convenient generator of $M_2$. This action arises from the action of $\TT$ on $\CC^2$ via the representation
\[
\TT\ni{z}\longmapsto\begin{bmatrix}z&0\\0&1\end{bmatrix}\in\B(\CC^2).
\]

\subsection{Identification $M_2\btens_\zeta{M_2}\cong{M_2}\tens{M_2}$}

For this we follow \cite[Proof of Theorem 5.4]{KMRW}, where the isomorphism $M_2\btens_\zeta{M_2}\cong{M_2}\tens{M_2}$ is explicitly constructed so that the canonical maps $\jmath_1,\jmath_2:M_2\to{M_2}\tens{M_2}$ are determined uniquely by the formula
\[
\jmath_1(S)\jmath_2(T)(\xi\tens\eta)=\overline{\zeta}^{\deg(T)\deg(\xi)}\jmath(S)\xi\tens{T}\eta
\]
valid for all $S\in{M_2}$, $\eta\in\CC^2$ and all homogeneous $T\in{M_2}$ and $\xi\in\CC^2$. It follows that $\jmath_1(S)=S\tens\I_2$ for all $S\in{M_2}$, while
\begin{align*}
\jmath_2(T)\left(\bigl[\begin{smallmatrix}1\\0\end{smallmatrix}\bigr]\tens\bigl[\begin{smallmatrix}1\\0\end{smallmatrix}\bigr]\right)&=\overline{\zeta}^{\deg(T)}\bigl[\begin{smallmatrix}1\\0\end{smallmatrix}\bigr]\tens{T}\bigl[\begin{smallmatrix}1\\0\end{smallmatrix}\bigr],\\
\jmath_2(T)\left(\bigl[\begin{smallmatrix}1\\0\end{smallmatrix}\bigr]\tens\bigl[\begin{smallmatrix}0\\1\end{smallmatrix}\bigr]\right)&=\overline{\zeta}^{\deg(T)}\bigl[\begin{smallmatrix}1\\0\end{smallmatrix}\bigr]\tens{T}\bigl[\begin{smallmatrix}0\\1\end{smallmatrix}\bigr],\\
\jmath_2(T)\left(\bigl[\begin{smallmatrix}0\\1\end{smallmatrix}\bigr]\tens\bigl[\begin{smallmatrix}1\\0\end{smallmatrix}\bigr]\right)&=\bigl[\begin{smallmatrix}0\\1\end{smallmatrix}\bigr]\tens{T}\bigl[\begin{smallmatrix}1\\0\end{smallmatrix}\bigr],\\
\jmath_2(T)\left(\bigl[\begin{smallmatrix}0\\1\end{smallmatrix}\bigr]\tens\bigl[\begin{smallmatrix}0\\1\end{smallmatrix}\bigr]\right)&=\bigl[\begin{smallmatrix}0\\1\end{smallmatrix}\bigr]\tens{T}\bigl[\begin{smallmatrix}0\\1\end{smallmatrix}\bigr].
\end{align*}
In other words $\jmath_2$ is determined uniquely by
\[
\jmath_2(\bn)=\begin{bmatrix}\overline{\zeta}&0\\0&1\end{bmatrix}\tens\bn.
\]

\subsection{The tensor square of the fundamental representation}

Using the above-mentioned identification $M_2\btens_\zeta{M_2}\cong{M_2}\tens{M_2}$ we will study the tensor square of the fundamental representation $\bu$ of $\SUq$. Recall from \cite{KMRW} that $\bu$ is the matrix
\[
\bu=\begin{bmatrix}\alpha&-q\gamma^*\\\gamma&\alpha^*\end{bmatrix}
\]
of elements of $\C(\SUq)$. It is a representation acting on the Hilbert space $\CC^2$ equipped with an action of $\TT$ defined by the representation
\begin{equation}\label{TC2}
\TT\ni{z}\longmapsto\begin{bmatrix}z&0\\0&1\end{bmatrix}\in\B(\CC^2).
\end{equation}
The fact that $\bu$ is a representation amounts to saying that $\bu$ is a unitary element of $M_2\tens\C(\SUq)$ invariant for the action of $\TT$:
\[
\bigl(\rho^{M_2}_z\tens\rho^{\C(\SUq)}_z\bigr)(\bu)=\bu,\qquad{z}\in\TT,
\]
where the action of $\TT$ on $M_2$ is given by conjugation with \eqref{TC2}, and denoting matrix elements of $\bu$ by $\{\bu_{i,j}\}_{i,j=1,2}$ we have
\[
\DSUq(\bu_{i,j})=\sum_k\jmath_1(\bu_{i,k})\jmath_2(\bu_{k,j})
\]
for all $i,j$ (cf.~\eqref{DeltaSUq}).

According to the definition given in \cite[Proposition 5.3]{KMRW} the tensor square of the fundamental representation $\bu$ of $\SUq$ is
\[
(\jmath_1\tens\id)(\bu)(\jmath_2\tens\id)(\bu)\in{M_2}\btens_\zeta{M_2}\tens\C(\SUq)\cong{M_2}\tens{M_2}\tens\C(\SUq).
\]
Writing $\bu$ as
\[
\bu=\bn\bn^*\tens\alpha-q\bn\tens\gamma^*+\bn^*\tens\gamma+\bn^*\bn\tens\alpha^*
\]
we immediately get
\begin{equation}\label{uu}
\begin{split}
\bu\,{\scriptstyle\text{\raisebox{0.8pt}{\textcircled{\raisebox{-1.7pt}{$\top$}}}}}\!\;\bu
&=\bigl(\bn\bn^*\tens\I_2\tens\alpha-\bn\tens\I_2\tens{q}\gamma^*+\bn^*\tens\I_2\tens\gamma+\bn^*\bn\tens\I_2\tens\alpha^*\bigr)\\
&\quad\times\bigl(\I_2\tens\bn\bn^*\tens\alpha-\bigl[\begin{smallmatrix}\overline{q}&0\\0&q\end{smallmatrix}
\bigr]\tens\bn\tens\gamma^*+\bigl[\begin{smallmatrix}\zeta&0\\0&1\end{smallmatrix}
\bigr]\tens\bn^*\tens\gamma+\I_2\tens\bn^*\bn\tens\alpha^*\bigr).
\end{split}
\end{equation}
Using the identification $M_2\tens{M_2}$ with $M_4$
\[
\begin{bmatrix}
a&b\\c&d
\end{bmatrix}\tens
\begin{bmatrix}
a'&b'\\c'&d'
\end{bmatrix}
=\begin{bmatrix}
aa'&ab'&ba'&bb'\\
ac'&ad'&bc'&bd'\\
ca'&cb'&da'&db'\\
cc'&cd'&dc'&dd'
\end{bmatrix}
\]
and the identification $M_4\tens\C(\SUq)\cong{M_4(\C(\SUq)}$ we can rewrite right-hand side \eqref{uu} as
\begin{align*}
\begin{bmatrix}
\alpha&0&-q\gamma^*&0\\
0&\alpha&0&-q\gamma^*\\
\gamma&0&\alpha^*&0\\
0&\gamma&0&\alpha^*
\end{bmatrix}
\begin{bmatrix}
\alpha&-\overline{q}\gamma^*&0&0\\
\zeta\gamma&\alpha^*&0&0\\0&0&\alpha&-q\gamma^*\\
0&0&\gamma&\alpha^*
\end{bmatrix}
&=
\begin{bmatrix}
\alpha^2&-\overline{q}\alpha\gamma^*&-q\gamma^*\alpha&q^2{\gamma^*}^2\\
\zeta\alpha\gamma&\alpha\alpha^*&-q\gamma^*\gamma&-q\gamma^*\alpha^*\\
\gamma\alpha&-\overline{q}\gamma\gamma^*&\alpha^*\alpha&-q\alpha^*\gamma^*\\
\zeta\gamma^2&\gamma\alpha^*&\alpha^*\gamma&{\alpha^*}^2
\end{bmatrix}\\
&=\begin{bmatrix}
\alpha^2&-\overline{q}\alpha\gamma^*&-\alpha\gamma^*&q^2{\gamma^*}^2\\
q\gamma\alpha&\alpha\alpha^*&-q\gamma^*\gamma&-q\gamma^*\alpha^*\\
\gamma\alpha&-\overline{q}\gamma\gamma^*&\alpha^*\alpha&-\gamma^*\alpha^*\\
\zeta\gamma^2&\gamma\alpha^*&\alpha^*\gamma&{\alpha^*}^2
\end{bmatrix}.
\end{align*}
This is a representation of $\SUq$ acting on the Hilbert space $\CC^4\cong\CC^2\tens\CC^2$ with the tensor product action of $\TT$, i.e.~given by the representation
\[
\TT\ni{z}\longmapsto
\begin{bmatrix}
z^2&0&0&0\\
0&z&0&0\\
0&0&z&0\\
0&0&0&1
\end{bmatrix}
\in\B(\CC^4).
\]

\subsection{The three-dimensional representation}

It is easy to see that the vector
\[
\begin{bmatrix}
0\\1\\-\overline{q}\\0
\end{bmatrix}
\]
is invariant for $\bu\,{\scriptstyle\text{\raisebox{0.8pt}{\textcircled{\raisebox{-1.7pt}{$\top$}}}}}\!\;\bu$ (cf.~\cite[Section 5]{KMRW}). In order to find the remaining three-dimensional component of $\bu\,{\scriptstyle\text{\raisebox{0.8pt}{\textcircled{\raisebox{-1.7pt}{$\top$}}}}}\!\;\bu$ we change the basis to
\[
\left\{
\tfrac{1}{\sigma}\left[\begin{smallmatrix}
0\\1\\-\overline{q}\\0
\end{smallmatrix}\right],\left[\begin{smallmatrix}
1\\0\\0\\0
\end{smallmatrix}\right],\tfrac{1}{\sigma}\left[\begin{smallmatrix}
0\\q\\1\\0
\end{smallmatrix}\right],\left[\begin{smallmatrix}
0\\0\\0\\1
\end{smallmatrix}\right]
\right\},
\]
where $\sigma=\sqrt{1+|q|^2}$. The representation $\bu\,{\scriptstyle\text{\raisebox{0.8pt}{\textcircled{\raisebox{-1.7pt}{$\top$}}}}}\!\;\bu$ in this basis has the matrix
{\allowdisplaybreaks
\begin{align*}
\tfrac{1}{\sigma^2}
\begin{bmatrix}
0&1&-q&0\\
\sigma&0&0&0\\
0&\overline{q}&1&0\\
0&0&0&\sigma
\end{bmatrix}
&
\begin{bmatrix}
\alpha^2&-\overline{q}\alpha\gamma^*&-\alpha\gamma^*&q^2{\gamma^*}^2\\
\zeta\alpha\gamma&\alpha\alpha^*&-q\gamma^*\gamma&-q\gamma^*\alpha^*\\
\gamma\alpha&-\overline{q}\gamma\gamma^*&\alpha^*\alpha&-\gamma^*\alpha^*\\
\zeta\gamma^2&\gamma\alpha^*&\alpha^*\gamma&{\alpha^*}^2
\end{bmatrix}\begin{bmatrix}
0&\sigma&0&0\\
1&0&q&0\\
-\overline{q}&0&1&0\\
0&0&0&\sigma
\end{bmatrix}\\
&=\tfrac{1}{\sigma^2}
\begin{bmatrix}
0&1&-q&0\\
\sigma&0&0&0\\
0&\overline{q}&1&0\\
0&0&0&\sigma
\end{bmatrix}
\begin{bmatrix}
0&\sigma\alpha^2&-\sigma^2\alpha\gamma^*&q^2\sigma{\gamma^*}^2\\
\I&q\sigma\gamma\alpha&q(\I-\sigma^2\gamma^*\gamma)&-q\sigma\gamma^*\alpha^*\\
-\overline{q}\I&\sigma\gamma\alpha&\I-\sigma^2\gamma^*\gamma&-\sigma\gamma^*\alpha^*\\
0&{\zeta}\sigma\gamma^2&\sigma^2\alpha^*\gamma&\sigma{\alpha^*}^2
\end{bmatrix}\\
&=\begin{bmatrix}
\I&0&0&0\\
0&
\alpha^2&-\sigma\alpha\gamma^*&q^2{\gamma^*}^2\\
0&\sigma\gamma\alpha&\I-\sigma^2\gamma^*\gamma&-\sigma\gamma^*\alpha^*\\
0&\zeta\gamma^2&\sigma\alpha^*\gamma&{\alpha^*}^2
\end{bmatrix}.
\end{align*} } 

Denote by $\cH$ the subspace $\operatorname{span}\{f_2,f_1,f_0\}\subset\CC^4$, where
\[
f_2=\left[\begin{smallmatrix}1\\0\\0\\0\end{smallmatrix}\right],\quad
f_1=\tfrac{1}{\sigma}\left[\begin{smallmatrix}0\\q\\1\\0\end{smallmatrix}\right],\quad
f_0=\left[\begin{smallmatrix}0\\0\\0\\1\end{smallmatrix}\right].
\]
Then $\cH$ becomes a $\TT$-Hilbert space with the action inherited from $\CC^4$ and $f_i$ is a vector of degree $i$ for $i=2,1,0$.

The matrix
\[
\boldsymbol{W}=\begin{bmatrix}
\alpha^2&-\sigma\alpha\gamma^*&q^2{\gamma^*}^2\\
\sigma\gamma\alpha&\I-\sigma^2\gamma^*\gamma&-\sigma\gamma^*\alpha^*\\
\zeta\gamma^2&\sigma\alpha^*\gamma&{\alpha^*}^2
\end{bmatrix}
\]
is then a unitary representation of the braided quantum groups $\SUq$ on $\cH$. For example to see that $\boldsymbol{W}$ is invariant under the action of $\TT$ we note that for each $z\in\TT$
\[
\bigl(\rho^{\B(\cH)}_z\tens\rho^{\C(\SUq)}_z\bigr)(\boldsymbol{W})
=\begin{bmatrix}
z^2&0&0\\0&z&0\\0&0&1
\end{bmatrix}
\begin{bmatrix}
\alpha^2&-\overline{z}\sigma\alpha\gamma^*&\overline{z}^2q^2{\gamma^*}^2\\
z\sigma\gamma\alpha&\I-\sigma^2\gamma^*\gamma&-\overline{z}\sigma\gamma^*\alpha^*\\
z^2\zeta\gamma^2&z\sigma\alpha^*\gamma&{\alpha^*}^2
\end{bmatrix}
\begin{bmatrix}
\overline{z}^2&0&0\\0&\overline{z}&0\\0&0&1
\end{bmatrix}=\boldsymbol{W}.
\]

By scaling the basis we obtain a different form of this representation:
\begin{align*}
\bV&=\begin{bmatrix}
q^{-1}&0&0\\
0&\sigma^{-1}&0\\
0&0&-1
\end{bmatrix}
\begin{bmatrix}
\alpha^2&-\sigma\alpha\gamma^*&q^2{\gamma^*}^2\\
\sigma\gamma\alpha&\I-\sigma^2\gamma^*\gamma&-\sigma\gamma^*\alpha^*\\
\zeta\gamma^2&\sigma\alpha^*\gamma&{\alpha^*}^2
\end{bmatrix}
\begin{bmatrix}
q&0&0\\
0&\sigma&0\\
0&0&-1
\end{bmatrix}\\
&=\begin{bmatrix}
\alpha^2&-\sigma^2\gamma^*\alpha&-q{\gamma^*}^2\\
\zeta\alpha\gamma&\I-\sigma^2\gamma^*\gamma&\gamma^*\alpha^*\\
-q\zeta\gamma^2&-\sigma^2\alpha^*\gamma&{\alpha^*}^2
\end{bmatrix}
\end{align*}
(when $q$ is real it coincides with the form used by Podle\'{s} in \cite[Page 194]{spheres} after conjugation by $\left[
\begin{smallmatrix}
0&0&1\\
0&1&0\\
1&0&0
\end{smallmatrix}
\right]$) and denoting the matrix elements of $\bV$ by
\[
\bV=
\begin{bmatrix}
\bv_{-1,1}&\bv_{-1,0}&\bv_{-1,1}\\
\bv_{0,1}&\bv_{0,0}&\bv_{0,1}\\
\bv_{1,1}&\bv_{1,0}&\bv_{1,1}
\end{bmatrix}
\]
we have
\begin{equation}\label{V_is_a_rep}
\DSUq(\bv_{i,j})=\sum_{k=-1}^1\jmath_1(\bv_{i,k})\jmath_2(\bv_{k,j}),\qquad{i,j}=-1,0,1
\end{equation}
which can be checked directly as in Section \ref{directcheck}.

\subsection{Irreducibility}\label{irr}

The representation $\bV$ is irreducible in the following sense: if $A\in{M_3}$ commutes with $\bV$ then $A=\lambda\I_3$ for some $\lambda\in\CC$. Indeed if
For any scalar matrix
\[
A=\begin{bmatrix}
a_{-1,-1}&a_{-1,0}&a_{-1,1}\\
a_{0,-1}&a_{0,0}&a_{0,1}\\
a_{1,-1}&a_{1,0}&a_{1,1}
\end{bmatrix}
\]
the matrix $A\bV$ is
\[
\resizebox{\textwidth}{!}{\ensuremath{\begin{bmatrix}
a_{-1,-1}\bv_{-1,-1}+a_{-1,0}\bv_{0,-1}+a_{-1,1}\bv_{1,-1}&a_{-1,-1}\bv_{-1,0}+a_{-1,0}\bv_{0,0}+a_{-1,1}\bv_{1,0}&a_{-1,-1}\bv_{-1,1}+a_{-1,0}\bv_{0-1}+a_{-1,1}\bv_{1,1}\\
a_{0,-1}\bv_{-1,-1}+a_{0,0}\bv_{0,-1}+a_{0,1}\bv_{1,-1}&a_{0,-1}\bv_{-1,0}+a_{0,0}\bv_{0,0}+a_{0,1}\bv_{1,0}&a_{0,-1}\bv_{-1,1}+a_{0,0}\bv_{0-1}+a_{0,1}\bv_{1,1}\\
a_{1,-1}\bv_{-1,-1}+a_{1,0}\bv_{0,-1}+a_{1,1}\bv_{1,-1}&a_{1,-1}\bv_{-1,0}+a_{1,0}\bv_{0,0}+a_{1,1}\bv_{1,0}&a_{1,-1}\bv_{-1,1}+a_{1,0}\bv_{0-1}+a_{1,1}\bv_{1,1}
\end{bmatrix}}}
\]
while $\bV{A}$ is
\[
\resizebox{\textwidth}{!}{\ensuremath{\displaystyle\begin{bmatrix}
a_{-1,-1}\bv_{-1,-1}+a_{0,-1}\bv_{-1,0}+a_{1,-1}\bv_{-1,1}&a_{-1,0}\bv_{-1,-1}+a_{0,0}\bv_{-1,0}+a_{1,0}\bv_{-1,1}&a_{-1,1}\bv_{-1,-1}+a_{0,1}\bv_{-1,0}+a_{1,1}\bv_{-1,1}\\
a_{-1,-1}\bv_{0,-1}+a_{0,-1}\bv_{0,0}+a_{1,-1}\bv_{0,1}&a_{-1,0}\bv_{0,-1}+a_{0,0}\bv_{0,0}+a_{1,0}\bv_{0,1}&a_{-1,1}\bv_{0,-1}+a_{0,1}\bv_{0,0}+a_{1,1}\bv_{0,1}\\
a_{-1,-1}\bv_{1,-1}+a_{0,-1}\bv_{1,0}+a_{1,-1}\bv_{1,1}&a_{-1,0}\bv_{1,-1}+a_{0,0}\bv_{1,0}+a_{1,0}\bv_{1,1}&a_{-1,1}\bv_{1,-1}+a_{0,1}\bv_{1,0}+a_{1,1}\bv_{1,1}
\end{bmatrix}}}
\]
Hence irreducibility of $\bV$ follows if we can prove that the matrix elements of $\bV$ are linearly independent. This can be achieved along the lines of \cite[Theorem 1.2]{twisted}. More precisely let $\pi$ be the representation of $\C(\SUq)$ on $\ell_2(\ZZ_+\times\ZZ)$ defined by
\[
\pi(\alpha)e_{n,k}=\sqrt{1-|q|^{2n}}e_{n-1,k},\qquad
\pi(\gamma)e_{n,k}=\overline{q}^ne_{n,k+1},
\]
where $\{e_{n,k}\}_{n\in\ZZ_+,k\in\ZZ}$ is the standard orthonormal basis of $\ell_2(\ZZ_+\times\ZZ)$. We have
{\allowdisplaybreaks
\begin{align*}
\pi(\bv_{-1,-1})e_{n,k}&=\sqrt{1-|q|^{2n}}\sqrt{1-|q|^{2n-2}}e_{n-2,k},\\
\pi(\bv_{-1,0})e_{n,k}&=-\sigma^2q^{n-1}\sqrt{1-|q|^{2n}}e_{n-1,k-1},\\
\pi(\bv_{-1,1})e_{n,k}&=-q^{2n+1}e_{n,k-2},\\
\pi(\bv_{0,-1})e_{n,k}&=\zeta\overline{q}^n\sqrt{1-|q|^{2n}}e_{n-1,k+1},\\
\pi(\bv_{0,0})e_{n,k}&=\bigl(1+s^2|q|^{2n}\bigr)e_{n,k},\\
\pi(\bv_{0,1})e_{n,k}&=q^{n+1}\sqrt{1-|q|^{2n+2}}e_{n+1,k-1},\\
\pi(\bv_{1,-1})e_{n,k}&=-q\zeta\overline{q}^{2n}e_{n,k+2},\\
\pi(\bv_{1,0})e_{n,k}&=-\sigma^2\overline{q}^n\sqrt{1-|q|^{2n+2}}e_{n+1,k+1}\\
\pi(\bv_{1,1})e_{n,k}&=\sqrt{1-|q|^{2n+2}}\sqrt{1-|q|^{2n+4}}e_{n+2,k}.
\end{align*}
}This shows that the linear map
\[
\operatorname{span}\bigl\{\bv_{i,j}\,\bigl.\bigr|\,i,j=-1,0,1\bigr\}\ni{x}\longmapsto\pi(x)e_{2,0}\in\ell_2(\ZZ_+\times\ZZ)
\]
has nine-dimensional range.

\subsection{The quotient sphere in terms of $\bV$}\label{inTerms}

Let us note that the \cst-algebra $\C(\SS^2_q)$ of continuous functions on the quotient sphere $\SS^2_q$ is te unital \cst-subalgebra of $\C(\SUq)$ generated by the elements in the middle column of the representation $\bV$. Moreover writing
\[
e_i=\bv_{i,0},\qquad{i}=-1,0,1
\]
we find that $e_i$ is a homogeneous element of degree $i$, ${e_i}^*=e_{-i}$ for all $i$ and
\[
\Gamma(e_i)=\sum_{k=-1}^{1}\jmath_1(\bv_{i,k})\jmath_2(e_k),\qquad{i}=-1,0,1.
\]

Consider some three elements $x_{-1},x_0,x_1\in\C(\SUq)$ such that
\[
\DSUq(x_i)=\sum_{k=-1}^1\jmath_1(\bv_{i,k})\jmath_2(x_k),\qquad{i}=-1,0,1.
\]
Then for each $i$
\[
x_i=(\id\btens_\zeta\eps)\DSUq(x_i)=\sum_{k=-1}^1\eps(x_k)\bv_{i,k}.
\]
This shows that $x\in\operatorname{span}\{\bv_{i,j}\,\left.\right|i,j=-1,0,1\}$. If we further assume that $x\in\C(\SS^2_q)$, i.e.~that $x$ is invariant under the right action $\sigma$ of $\TT$, then $x$ has to be a linear combination of those matrix elements of $\bV$ which are invariant under $\sigma$ (cf.~the proof of Corollary \ref{genXq}). These are precisely $e_{-1},e_0$ and $e_1$. It follows that the only three-dimensional subspace of $\C(\SS^2_q)$ possessing a basis $x_{-1},x_0,x_1$ such that
\[
\Gamma(x_i)=\sum_{k=-1}^{1}\jmath_1(\bv_{i,k})\jmath_2(x_k),\qquad{i}=-1,0,1
\]
is $\operatorname{span}\{e_{-1},e_0,e_1\}$.

\section{Quantum spaces with a simple action of braided $\SUq$}\label{simplespaces}

Let $\XX$ be a compact quantum space such that $\C(\XX)\in\operatorname{Ob}(\CstT)$ and there exists $\Gamma:\CC(\XX)\to\C(\SUq)\btens_\zeta\C(\XX)$ such that
\[
(\Gamma\btens_\zeta\id)\comp\Gamma=(\id\btens_\zeta\DSUq)\comp\Gamma.
\]
Assume that
\begin{enumerate}
\item\label{simple1} if $a\in\C(\XX)$ satisfies $\Gamma(a)=\jmath_2(a)$ then $a\in\CC\I$,
\item\label{simple2} $\C(\XX)$ is generated by $\I$ and a three-dimensional subspace $\mathsf{W}\subset\C(\XX)$ equipped with a basis $\{e_{-1},e_0,e_1\}$ such that
\[
\Gamma(e_i)=\sum_{k=-1}^1\jmath_1(\bv_{i,k})\jmath_2(e_k),\qquad{i}=-1,0,1
\]
(we say that the basis $\{e_{-1},e_0,e_1\}$ \emph{transforms according} to $\bV$),
\item\label{simple2.5} $\mathsf{W}$ is the only three-dimensional subspace of $\C(\XX)$ equipped with a basis which transforms according to $\bV$,
\item\label{simple3} for each $i$ the element $e_i$ is homogeneous of degree $i$ under the action of $\TT$ on $\C(\XX)$.
\end{enumerate}

\begin{remark}
The results of Theorem \ref{quotientsphere} and Section \ref{inTerms} show that the quotient sphere $\SS^2_q$ described in Section \ref{theQuotientS2} satisfies the conditions \eqref{simple1}--\eqref{simple3}.
\end{remark}

\begin{proposition}\label{uniquebasis}
$\{e_{-1},e_0,e_1\}$ is unique up to proportionality basis of $\mathsf{W}$ which satisfies condition \eqref{simple2}.
\end{proposition}

\begin{proof}
Let $\{f_{-1},f_0,f_1\}$ be another basis of $\mathsf{W}$ with the property
\[
\Gamma(f_i)=\sum_{i=-1}^1\jmath_1(\bv_{i,j})\jmath_2(f_j),\qquad{i}=-1,0,1
\]
There is an invertible $3\times{3}$ scalar matrix $A$ such that
\[
\begin{bmatrix}f_{-1}\\f_0\\f_1\end{bmatrix}=A\begin{bmatrix}e_{-1}\\e_0\\e_1\end{bmatrix}.
\]
Therefore the condition
\[
\begin{bmatrix}\Gamma(f_{-1})\\\Gamma(f_0)\\\Gamma(f_1)\end{bmatrix}
=\begin{bmatrix}
\jmath_1(\bv_{-1,-1})&\jmath_1(\bv_{-1,0})&\jmath_1(\bv_{-1,1})\\
\jmath_1(\bv_{0,-1})&\jmath_1(\bv_{0,0})&\jmath_1(\bv_{0,1})\\
\jmath_1(\bv_{1,-1})&\jmath_1(\bv_{1,0})&\jmath_1(\bv_{1,1})
\end{bmatrix}
\begin{bmatrix}\jmath_2(f_{-1})\\\jmath_2(f_0)\\\jmath_2(f_1)\end{bmatrix}
\]
means
\[
A\begin{bmatrix}\Gamma(e_{-1})\\\Gamma(e_0)\\\Gamma(e_1)\end{bmatrix}
=\begin{bmatrix}
\jmath_1(\bv_{-1,-1})&\jmath_1(\bv_{-1,0})&\jmath_1(\bv_{-1,1})\\
\jmath_1(\bv_{0,-1})&\jmath_1(\bv_{0,0})&\jmath_1(\bv_{0,1})\\
\jmath_1(\bv_{1,-1})&\jmath_1(\bv_{1,0})&\jmath_1(\bv_{1,1})
\end{bmatrix}
A\begin{bmatrix}\jmath_2(e_{-1})\\\jmath_2(e_0)\\\jmath_2(e_1)\end{bmatrix}
\]
or in other words
\[
\begin{bmatrix}\Gamma(e_{-1})\\\Gamma(e_0)\\\Gamma(e_1)\end{bmatrix}
=A^{-1}\begin{bmatrix}
\jmath_1(\bv_{-1,-1})&\jmath_1(\bv_{-1,0})&\jmath_1(\bv_{-1,1})\\
\jmath_1(\bv_{0,-1})&\jmath_1(\bv_{0,0})&\jmath_1(\bv_{0,1})\\
\jmath_1(\bv_{1,-1})&\jmath_1(\bv_{1,0})&\jmath_1(\bv_{1,1})
\end{bmatrix}
A\begin{bmatrix}\jmath_2(e_{-1})\\\jmath_2(e_0)\\\jmath_2(e_1)\end{bmatrix}.
\]

Let us write $\widetilde{\bv}_{i,j}$ for the $(i,j)$-element of the matrix $A^{-1}\bV{A}$. Then the above shows that for each $i$
\begin{equation}\label{rownedzialania}
\sum_k\jmath_1(\widetilde{\bv}_{i,k})\jmath_2(e_k)=\sum_k\jmath_1(\bv_{i,k})\jmath_2(e_k).
\end{equation}
Now we note that there is an action $\varrho$ of $\TT$ on $\C(\SUq)\btens\C(\XX)$ such that
\[
\varrho_z\bigl(\jmath_1(a)\jmath_2(x)\bigr)=\jmath_1(a)\jmath\bigl(\rho^{\C(\XX)}_z(x)\bigr),\qquad{z}\in\TT
\]
which is the restriction of the action of $\TT^2$ described in Remark \ref{actionT2} to the subgroup $\bigl\{(1,z)\,\bigl.\bigr|\,z\in\TT\bigr\}$. Multiplying both sides of \eqref{rownedzialania} by $z^{-l}$, applying $\varrho_z$ to both sides and integrating over $z\in\TT$ we obtain
\[
\jmath_1(\widetilde{\bv}_{i,l})=\jmath_1(\bv_{i,l})
\]
for $l=-1,0,1$. Since this holds for all $i$ and the morphism $\jmath_1$ is injective (\cite[Section 1]{KMRW}), it follows that
\[
\bV=A^{-1}\bV{A},
\]
so $A$ must be proportional to $\I_3$ and consequently there exists a non-zero scalar $\lambda$ such that $f_i=\lambda{e_i}$ for all $i$.
\end{proof}

\begin{proposition}
There exists a non-zero complex number $\delta$ such that ${e_i}^*={\delta}e_{-i}$ for $i=-1,0,1$.
\end{proposition}

\begin{proof}
We first note that
\begin{align*}
{\bv_{-1,-1}}^*&=\bv_{1,1},&{\bv_{-1,0}}^*&=\bv_{1,0},&{\bv_{0,0}}^*&=\bv_{0,0},\\
{\bv_{0,-1}}^*&=\overline{\zeta}\bv_{0,1},&&&{\bv_{1,-1}}^*&=\overline{\zeta}^2\bv_{-1,1}
\end{align*}
and the degrees of matrix elements of $\bV$ are
\begin{equation}\label{degrees}
\begin{bmatrix}
\operatorname{deg}(\bv_{-1,-1})&\operatorname{deg}(\bv_{-1,0})&\operatorname{deg}(\bv_{-1,1})\\
\operatorname{deg}(\bv_{0,-1})&\operatorname{deg}(\bv_{0,0})&\operatorname{deg}(\bv_{0,1})\\
\operatorname{deg}(\bv_{1,-1})&\operatorname{deg}(\bv_{1,0})&\operatorname{deg}(\bv_{1,1})
\end{bmatrix}
=
\begin{bmatrix}
0&-1&-2\\
1&0&-1\\
2&1&0
\end{bmatrix}.
\end{equation}

Using this and formula \eqref{braid} we compute
\begin{align*}
\Gamma({e_{-1}}^*)&=\bigl(\jmath_1(\bv_{-1,-1})\jmath_2(e_{-1})+\jmath_1(\bv_{-1,0})\jmath_2(e_0)+\jmath_1(\bv_{-1,1})\jmath_2(e_1)\bigr)^*\\
&=\jmath_2({e_{-1}}^*)\jmath_1(\bv_{1,1})+\jmath_2({e_0}^*)\jmath_1(\bv_{1,0})+\overline{\zeta}^2\jmath_2({e_1}^*)\jmath_1(\bv_{1,-1})\\
&=\overline{\zeta}^{1\cdot{0}}\jmath_1(\bv_{1,1})\jmath_2({e_{-1}}^*)+\overline{\zeta}^{0\cdot{1}}\jmath_1(\bv_{1,0})\jmath_2({e_0}^*)+\overline{\zeta}^2\overline{\zeta}^{(-1)\cdot{2}}\jmath_1(\bv_{1,-1})\jmath_2({e_1}^*)\\
&=\jmath_1(\bv_{1,1})\jmath_2({e_{-1}}^*)+\jmath_1(\bv_{1,0})\jmath_2({e_0}^*)+\jmath_1(\bv_{1,-1})\jmath_2({e_1}^*).
\end{align*}
Similarly
\begin{align*}
\Gamma({e_0}^*)&=\bigl(\jmath_1(\bv_{0,-1})\jmath_2(e_{-1})+\jmath_1(\bv_{0,0})\jmath_2(e_0)+\jmath_1(\bv_{0,1})\jmath_2(e_1)\bigr)^*\\
&=\overline{\zeta}\jmath_2({e_{-1}}^*)\jmath_1(\bv_{0,1})+\jmath_2({e_0}^*)\jmath_1(\bv_{0,0})+\overline{\zeta}\jmath_2({e_1}^*)\jmath_1(\bv_{0,-1})\\
&=\overline{\zeta}\,\overline{\zeta}^{1\cdot(-1)}\jmath_2({e_{-1}}^*)\jmath_1(\bv_{0,1})+\overline{\zeta}^{0\cdot{0}}\jmath_2({e_0}^*)\jmath_1(\bv_{0,0})
+\overline{\zeta}\,\overline{\zeta}^{(-1)\cdot{1}}\jmath_2({e_1}^*)\jmath_1(\bv_{0,-1})\\
&=\jmath_2({e_{-1}}^*)\jmath_1(\bv_{0,1})+\jmath_2({e_0}^*)\jmath_1(\bv_{0,0})+\jmath_2({e_1}^*)\jmath_1(\bv_{0,-1})
\end{align*}
and
\begin{align*}
\Gamma({e_1}^*)&=\bigl(\jmath_1(\bv_{1,-1})\jmath_2(e_{-1})+\jmath_1(\bv_{1,0})\jmath_2(e_0)+\jmath_1(\bv_{1,1})\jmath_2(e_1)\bigr)^*\\
&=\overline{\zeta}^2\jmath_2({e_{-1}}^*)\jmath_1(\bv_{-1,1})+\jmath_2({e_0}^*)\jmath_1(\bv_{-1,0})+\jmath_2({e_1}^*)\jmath_1(\bv_{-1,-1})\\
&=\overline{\zeta}^2\overline{\zeta}^{1\cdot(-2)}\jmath_1(\bv_{-1,1})\jmath_2({e_{-1}}^*)+\overline{\zeta}^{0\cdot(-1)}\jmath_1(\bv_{-1,0})\jmath_2({e_0}^*)+\overline{\zeta}^{(-1)\cdot{0}}\jmath_1(\bv_{-1,-1})\jmath_2({e_1}^*)\\
&=\jmath_1(\bv_{-1,1})\jmath_2({e_{-1}}^*)+\jmath_1(\bv_{-1,0})\jmath_2({e_0}^*)+\jmath_1(\bv_{-1,-1})\jmath_2({e_1}^*).
\end{align*}

It follows that the basis $\bigl\{{e_1}^*,{e_0}^*,{e_{-1}}^*\bigr\}$ satisfies the conditions fixing the basis $\{e_{-1},e_0,e_1\}$ uniquely up to proportionality, so there exists a non-zero $\delta\in\CC$ such that ${e_i}^*=\delta{e_{-i}}$ for all $i\in\{-1,0,1\}$.
\end{proof}

Now since $e_0=\bigl({e_0}^*\bigr)^*=\bigl(\delta{e_0}\bigr)^*=\overline{\delta}{e_0}^*=|\delta|^2e_0$, so $|\delta|=1$. Writing $\delta=\mathrm{e}^{\mathrm{i}\theta}$ and putting $f_i=\mathrm{e}^{-\mathrm{i}\frac{\theta}{2}}e_i$ for $i=-1,0,1$ we obtain ${f_i}^*=f_{-i}$ for all $i$. In other words, replacing $e_i$ with $\mathrm{e}^{-\mathrm{i}\frac{\theta}{2}}e_i$ we can assume that
\[
{e_i}^*=e_{-i},\qquad{i}=-1,0,1.
\]

\begin{proposition}
For $k,l\in\{-1,0,1\}$ we have
\begin{equation}\label{Gekel}
\Gamma(e_ke_l)=\sum_{r,p}\zeta^{r(p-l)}\jmath_1(\bv_{k,r}\bv_{l,p})\jmath_2(e_re_p).
\end{equation}
\end{proposition}

\begin{proof}
Using the fact that $\deg(e_r)=r$ and $\deg(\bv_{l,p})=p-l$ (cf.~\eqref{degrees}) and formula \eqref{braid} we compute
\begin{align*}
\Gamma(e_ke_l)&=\Gamma(e_k)\Gamma(e_l)=\biggl(\sum_r\jmath_1(\bv_{k,r})\jmath_2(e_r)\biggr)\biggl(\sum_p\jmath_1(\bv_{l,p})\jmath_2(e_p)\biggr)\\
&=\sum_{r,p}\jmath_1(\bv_{k,r})\jmath_2(e_r)\jmath_1(\bv_{l,p})\jmath_2(e_p)\\
&=\sum_{r,p}\jmath_1(\bv_{k,r})\overline{\zeta}^{r(l-p)}\jmath_1(\bv_{l,p})\jmath_2(e_r)\jmath_2(e_p)\\
&=\sum_{r,p}\zeta^{r(p-l)}\jmath_1(\bv_{k,r}\bv_{l,p})\jmath_2(e_re_p).
\end{align*}
\end{proof}

\begin{remark}\label{bigV}
Let us introduce an order on pairs $\bigl\{(k,l)\,\bigl.\bigr|\,k,l\in\{-1,0,1\}\bigr\}$ by pulling back the standard order on $\{1,\dotsc,9\}$ through the bijection
\[
(k,l)\longmapsto{3^{k+1}+(l+1)}.
\]
(in other words $(k+1)(l+1)$ is treated as an expression of a number form $\{1,\dotsc,9\}$ in the ternary numeral system; this order is compatible with our conventions on tensor products of matrices). Therefore, given $k,l\in\{-1,0,1\}$ and a vector with $9$ components it makes sense to talk about the $(k,l)$-component of the vector.

It follows from formula \eqref{Gekel} that $\Gamma(e_ke_l)$ is equal to the $(k,l)$-component of the result of matrix multiplication of the matrix
\[
\resizebox{\textwidth}{!}{\ensuremath{\VV=\begin{bmatrix}
\bv_{-1,-1}\bv_{-1,-1}&\overline{\zeta}\bv_{-1,-1}\bv_{-1,0}&\overline{\zeta}^2\bv_{-1,-1}\bv_{-1,1}&\bv_{-1,0}\bv_{-1,-1}&\bv_{-1,0}\bv_{-1,0}&\bv_{-1,0}\bv_{-1,1}&\bv_{-1,1}\bv_{-1,-1}&\zeta\bv_{-1,1}\bv_{-1,0}&\zeta^2\bv_{-1,1}\bv_{-1,1}\\
\zeta\bv_{-1,-1}\bv_{0,-1}&\bv_{-1,-1}\bv_{0,0}&\overline{\zeta}\bv_{-1,-1}\bv_{0,1}&\bv_{-1,0}\bv_{0,-1}&\bv_{-1,0}\bv_{0,0}&\bv_{-1,0}\bv_{0,1}&\overline{\zeta}\bv_{-1,1}\bv_{0,-1}&\bv_{-1,1}\bv_{0,0}&\zeta\bv_{-1,1}\bv_{0,1}\\
\zeta^2\bv_{-1,-1}\bv_{1,-1}&\zeta\bv_{-1,-1}\bv_{1,0}&\bv_{-1,-1}\bv_{1,1}&\bv_{-1,0}\bv_{1,-1}&\bv_{-1,0}\bv_{1,0}&\bv_{-1,0}\bv_{1,1}&\overline{\zeta}^2\bv_{-1,1}\bv_{1,-1}&\overline{\zeta}\bv_{-1,1}\bv_{1,0}&\bv_{-1,1}\bv_{1,1}\\
\bv_{0,-1}\bv_{-1,-1}&\overline{\zeta}\bv_{0,-1}\bv_{-1,0}&\overline{\zeta}^2\bv_{0,-1}\bv_{-1,1}&\bv_{0,0}\bv_{-1,-1}&\bv_{0,0}\bv_{-1,0}&\bv_{0,0}\bv_{-1,1}&\bv_{0,1}\bv_{-1,-1}&\zeta\bv_{0,1}\bv_{-1,0}&\zeta^2\bv_{0,1}\bv_{-1,1}\\
\zeta\bv_{0,-1}\bv_{0,-1}&\bv_{0,-1}\bv_{0,0}&\overline{\zeta}\bv_{0,-1}\bv_{0,1}&\bv_{0,0}\bv_{0,-1}&\bv_{0,0}\bv_{0,0}&\bv_{0,0}\bv_{0,1}&\overline{\zeta}\bv_{0,1}\bv_{0,-1}&\bv_{0,1}\bv_{0,0}&\zeta\bv_{0,1}\bv_{0,1}\\
\zeta^2\bv_{0,-1}\bv_{1,-1}&\zeta\bv_{0,-1}\bv_{1,0}&\bv_{0,-1}\bv_{1,1}&\bv_{0,0}\bv_{1,-1}&\bv_{0,0}\bv_{1,0}&\bv_{0,0}\bv_{1,1}&\overline{\zeta}^2\bv_{0,1}\bv_{1,-1}&\overline{\zeta}\bv_{0,1}\bv_{1,0}&\bv_{0,1}\bv_{1,1}\\
\bv_{1,-1}\bv_{-1,-1}&\overline{\zeta}\bv_{1,-1}\bv_{-1,0}&\overline{\zeta}^2\bv_{1,-1}\bv_{-1,1}&\bv_{1,0}\bv_{-1,-1}&\bv_{1,0}\bv_{-1,0}&\bv_{1,0}\bv_{-1,1}&\bv_{1,1}\bv_{-1,-1}&\zeta\bv_{1,1}\bv_{-1,0}&\zeta^2\bv_{1,1}\bv_{-1,1}\\
\zeta\bv_{1,-1}\bv_{0,-1}
&\bv_{1,-1}\bv_{0,0}
&\overline{\zeta}\bv_{1,-1}\bv_{0,1}
&\bv_{1,0}\bv_{0,-1}
&\bv_{1,0}\bv_{0,0}
&\bv_{1,0}\bv_{0,1}
&\overline{\zeta}\bv_{1,1}\bv_{0,-1}
&\bv_{1,1}\bv_{0,0}
&\zeta\bv_{1,1}\bv_{0,1}\\
\zeta^2\bv_{1,-1}\bv_{1,-1}
&\zeta\bv_{1,-1}\bv_{1,0}
&\bv_{1,-1}\bv_{1,1}
&\bv_{1,0}\bv_{1,-1}
&\bv_{1,0}\bv_{1,0}
&\bv_{1,0}\bv_{1,1}
&\overline{\zeta}^2\bv_{1,1}\bv_{1,-1}
&\overline{\zeta}\bv_{1,1}\bv_{1,0}
&\bv_{1,1}\bv_{1,1}
\end{bmatrix}}}
\]
by the vector
\[
\begin{bmatrix}
e_{-1}e_{-1}\\e_{-1}e_{1}\\e_{-1}e_{0}\\e_{0}e_{-1}\\e_{0}e_{0}\\e_{0}e_{1}\\e_{1}e_{-1}\\e_{1}e_{0}\\e_{1}e_{1}\\
\end{bmatrix}
\]
with the proviso that we apply $\jmath_1$ to elements of the matrix and $\jmath_2$ to components of the vector.
\end{remark}

In what follows, apart from $\sigma=\sqrt{1+|q|^2}$ we will also use the shorthand $\varsigma=|q|^2$.

\begin{proposition}\label{onP0}
There exists a constant $\rho\in\RR$ such that $e_{-1}e_1+\sigma^2e_0^2+{\varsigma}e_1e_{-1}=\rho\I$.
\end{proposition}

\begin{proof}
By Remark \ref{bigV} the quantity $\Gamma\bigl(\sum\lambda^{k,l}e_ke_l\bigr)$ is the matrix product
\[
\begin{bmatrix}\lambda^{-1,-1}&\lambda^{-1,0}&\dotsm&\lambda^{1,0}&\lambda^{1,1}\end{bmatrix}
\bigl((\id_{M_9}\tens\jmath_1)\VV\bigr)
\begin{bmatrix}
\jmath_2(e_{-1}e_{-1})\\\jmath_2(e_{-1}e_{1})\\\vdots\\\jmath_2(e_{1}e_{0})\\\jmath_2(e_{1}e_{1})\\
\end{bmatrix}.
\]
Using the calculations in Section \ref{vv} we immediately find that
\[
\begin{bmatrix}
0&0&1&0&\sigma^2&0&\varsigma&0&0
\end{bmatrix}
\VV=
\begin{bmatrix}
0&0&\I&0&\sigma^2\I&0&\varsigma\I&0&0
\end{bmatrix}
\]
which shows that the element $x=e_{-1}e_1+\sigma^2e_0^2+{\varsigma}e_1e_{-1}$ satisfies $\Gamma(x)=\jmath_2(x)$. Therefore, by condition \eqref{simple1}, $x$ must be proportional to $\I$, i.e.~there is a constant $\rho$ such that $x=\rho\I$. As $x$ is self-adjoint, we obtain $\rho\in\RR$.
\end{proof}

\begin{proposition}\label{onPi}
Let
\[
P_{-1}=\sigma^2e_{-1}e_0-\varsigma\sigma^2e_0e_{-1},\quad
P_0=\varsigma(e_1e_{-1}-e_{-1}e_1)+(1-\varsigma^2)e_0^2,\quad
P_1=\sigma^2e_0e_1-\varsigma\sigma^2e_1e_0.
\]
Then
\begin{equation}\label{act_on_P}
\Gamma(P_i)=\sum_k\jmath_1(\bv_{i,k})\jmath_2(P_k),\qquad{i}=-1,0,1.
\end{equation}
\end{proposition}

\begin{proof}
Using expression of matrix elements of $\VV$ as linear combinations of $\bigl\{a_{n,k,l}\}_{n\in\ZZ,k,l\in\ZZ_+}$ as well as the fact that
\[
\bV=
\begin{bmatrix}
\bv_{-1,-1}&\bv_{-1,0}&\bv_{-1,1}\\
\bv_{0,-1}&\bv_{0,0}&\bv_{0,1}\\
\bv_{1,-1}&\bv_{1,0}&\bv_{1,1}
\end{bmatrix}
=
\begin{bmatrix}
a_{2,0,0}&-q^{-1}\sigma^2a_{1,0,1}&-qa_{0,0,2}\\
{\zeta}a_{1,1,0}&\I-\sigma^2a_{0,1,1}&qa_{-1,0,1}\\
-q{\zeta}a_{0,2,0}&-\sigma^2a_{-1,1,0}&a_{-2,0,0}
\end{bmatrix}
\]
we find that with
\[
A=\sigma^2,{\quad}
B=-\varsigma,{\quad}
C=-\varsigma\sigma^2,{\quad}
D=1-\varsigma^2,{\quad}
E=\sigma^2,{\quad}
F=\varsigma,{\quad}
G=-\varsigma\sigma^2
\]
\begin{itemize}
\item the product $\begin{bmatrix}0&A&0&C&0&0&0&0&0\end{bmatrix}\VV$ is
\[
\begin{bmatrix}0&A\bv_{-1,-1}&B\bv_{-1,0}&C\bv_{-1,-1}&D\bv_{-1,0}&E\bv_{-1,1}&F\bv_{-1,0}&G\bv_{-1,1}&0\end{bmatrix},
\]
\item the product $\begin{bmatrix}0&0&B&0&D&0&F&0&0\end{bmatrix}\VV$ is
\[
\begin{bmatrix}0&A\bv_{0,-1}&B\bv_{0,0}&C\bv_{0,-1}&D\bv_{0,0}&E\bv_{0,1}&F\bv_{0,0}&G\bv_{0,1}&0\end{bmatrix},
\]
\item the product $\begin{bmatrix}0&0&0&0&0&E&0&G&0\end{bmatrix}\VV$ is
\[
\begin{bmatrix}0&A\bv_{1,-1}&B\bv_{1,0}&C\bv_{1,-1}&D\bv_{1,0}&E\bv_{1,1}&F\bv_{1,0}&G\bv_{1,1}&0\end{bmatrix},
\]
\end{itemize}
which in view of Remark \ref{bigV} (cf.~the proof of Proposition \ref{onP0}) is exactly \eqref{act_on_P}.
\end{proof}

\begin{corollary}
There exists a real number $\lambda$ such that $P_i={\lambda}e_i$ for $i=-1,0,1$.
\end{corollary}

\begin{proof}
It follows immediately from condition \eqref{simple2.5} and the uniqueness of the basis $\{e_{-1},e_0,e_1\}$ described in Proposition \ref{uniquebasis} that there is a $\lambda\in\CC$ such that $P_i=\lambda{e_i}$ for all $i$. Since $P_0$ and $e_0$ are self-adjoint, we must have $\lambda\in\RR$.
\end{proof}

\section{Braided Podle\'s spheres}\label{bps}

\begin{theorem}
Let $\XX_{q,\rho,\lambda}$ be the compact quantum space such that $\C(\XX_{q,\rho,\lambda})$ is the universal unital \cst-algebra generated by three elements $e_{-1},e_0,e_1$ with relations ${e_i}^*=e_{-i}$ for all $i$ and
\begin{align*}
e_{-1}e_1+\sigma^2e_0^2+{\varsigma}e_1e_{-1}&=\rho\I,\\
\sigma^2e_{-1}e_0-\varsigma\sigma^2e_0e_{-1}&={\lambda}e_{-1},\\
\varsigma(e_1e_{-1}-e_{-1}e_1)+(1-\varsigma^2)e_0^2&={\lambda}e_0,\\
\sigma^2e_0e_1-\varsigma\sigma^2e_1e_0&={\lambda}e_1.
\end{align*}
Then
\begin{enumerate}
\item\label{TactionXX} there is an action of $\TT$ on $\C(\XX_{q,\rho,\lambda})$ such that $e_i$ is of degree $i$,
\item\label{actionXX} there exists $\Gamma_{q,\rho,\lambda}\in\Mor_\TT(\C(\XX_{q,\rho,\lambda}),\C(\SUq)\btens_\zeta\C(\XX_{q,\rho,\lambda}))$ such that
\[
\Gamma_{q,\rho,\lambda}(e_i)=\sum_{k=-1}^1\jmath_1(\bv_{i,k})\jmath_2(e_k),\qquad{i}=-1,0,1,
\]
\item\label{gestoscXX} $\Gamma_{q,\rho,\lambda}$ satisfies $(\id\btens_\zeta\Gamma_{q,\rho,\lambda})\comp\Gamma_{q,\rho,\lambda}=(\DSUq\btens_\zeta\id)\comp\Gamma_{q,\rho,\lambda}$ and we have
\[
\bigl[\jmath_1(\C(\SUq))\Gamma_{q,\rho,\lambda}(\C(\XX_{q,\rho,\lambda}))\bigr]=\C(\SUq)\btens_\zeta\C(\XX_{q,\rho,\lambda}),
\]
\item\label{PodlesXX} if $\lambda'=\tfrac{\lambda}{\sqrt{\varsigma}\sigma^2}$ and $\rho'=\tfrac{\rho}{\varsigma\sigma^2}$ then $\C(\XX_{q,\rho,\lambda})$ is isomorphic to the \cst-algebra $\C(X_{|q|\lambda'\rho'})$ defined by Podle\'s in \cite[Section 3]{spheres}.
\end{enumerate}
\end{theorem}

\begin{proof}
Ad \eqref{TactionXX}. It is obvious that for any $z\in\TT$ the elements $\overline{z}e_{-1},e_0,ze_1$ satisfy the defining relations of $\C(\XX_{q,\rho,\lambda})$, so by universality there exists an automorphism of $\C(\XX_{q,\rho,\lambda})$ mapping $e_i$ to $z^{i}e_i$. Standard argument shows that this yields an continuous action of $\TT$ on $\C(\XX_{q,\rho,\lambda})$ as desired.

Ad \eqref{actionXX}. In order to show that $\Gamma_{q,\rho,\lambda}$ exists we need to prove that elements $P_{-1},P_0,P_1\in\C(\SUq)\btens_\zeta\C(\XX_{q,\rho,\lambda})$ defined by
\[
P_i=\sum_{k=-1}^1\jmath_1(\bv_{i,k})\jmath_2(e_k),\qquad{i}=-1,0,1
\]
satisfy ${P_i}^*=P_{-i}$ for all $i$ and
\begin{subequations}
\begin{align}
P_{-1}P_1+\sigma^2P_0^2+{\varsigma}P_1P_{-1}&=\rho\I,\label{P}\\
\sigma^2P_{-1}P_0-\varsigma\sigma^2P_0P_{-1}&={\lambda}P_{-1},\label{P-1}\\
\varsigma(P_1P_{-1}-P_{-1}P_1)+(1-\varsigma^2)P_0^2&={\lambda}P_0,\label{P0}\\
\sigma^2P_0P_1-\varsigma\sigma^2P_1P_0&={\lambda}P_1.\label{P1}
\end{align}
\end{subequations}

The left-hand sides of the above equations can be calculated using the commutation relation \eqref{braid}, e.g.
{\allowdisplaybreaks
\begin{align*}
P_{-1}P_1+\sigma^2P_0^2+{\varsigma}P_1P_{-1}
&=
\biggl(\sum_{r=-1}^1\jmath_1(\bv_{-1,r})\jmath_2(e_r)\biggr)\biggl(\sum_{p=-1}^1\jmath_1(\bv_{1,p})\jmath_2(e_p)\biggr)\\
&\quad+\sigma^2\biggl(\sum_r\jmath_1(\bv_{0,r})\jmath_2(e_r)\biggr)\biggl(\sum_p\jmath_1(\bv_{0,p})\jmath_2(e_p)\biggr)\\
&\quad+\varsigma\biggl(\sum_r\jmath_1(\bv_{1,r})\jmath_2(e_r)\biggr)\biggl(\sum_p\jmath_1(\bv_{-1,p})\jmath_2(e_p)\biggr)\\
&=\sum_{r,p}\zeta^{r(p-1)}\jmath_1(\bv_{-1,r}\bv_{1,p})\jmath_2(e_re_p)\\
&\quad+\sigma^2\sum_{r,p}\zeta^{rp}\jmath_1(\bv_{0,r}\bv_{0,p})\jmath_2(e_re_p)\\
&\quad+\varsigma\sum_{r,p}\zeta^{r(p+1)}\jmath_1(\bv_{1,r}\bv_{-1,p})\jmath_2(e_re_p).
\end{align*}
}The right-hand side of the above can be written in the short form of:
\[
\begin{bmatrix}
0&0&1&0&\sigma^2&0&\varsigma&0&0
\end{bmatrix}
\bigl((\id_{M_9}\tens\jmath_1)\VV\bigr)
\begin{bmatrix}
\jmath_2(e_{-1}e_{-1})\\
\jmath_2(e_{-1}e_{0})\\
\jmath_2(e_{-1}e_{1})\\
\jmath_2(e_{0}e_{-1})\\
\jmath_2(e_{0}e_{0})\\
\jmath_2(e_{0}e_{1})\\
\jmath_2(e_{1}e_{-1})\\
\jmath_2(e_{1}e_{0})\\
\jmath_2(e_{1}e_{1})
\end{bmatrix}
\]
which, as in the proof of Proposition \ref{onP0}, is
\[
\begin{bmatrix}
0&0&\I&0&\sigma^2\I&0&\varsigma\I&0&0
\end{bmatrix}
\begin{bmatrix}
\jmath_2(e_{-1}e_{-1})\\
\jmath_2(e_{-1}e_{0})\\
\jmath_2(e_{-1}e_{1})\\
\jmath_2(e_{0}e_{-1})\\
\jmath_2(e_{0}e_{0})\\
\jmath_2(e_{0}e_{1})\\
\jmath_2(e_{1}e_{-1})\\
\jmath_2(e_{1}e_{0})\\
\jmath_2(e_{1}e_{1})
\end{bmatrix}=\jmath_2(e_{-1}e_1+\sigma^2e_0^2+{\varsigma}e_1e_{-1})=\rho\I
\]
and this means that relation \eqref{P} holds.

Similarly the calculations performed in the proof of Proposition \ref{onPi} show that
\begin{align*}
\sigma^2P_{-1}P_0&-\varsigma\sigma^2P_0P_{-1}=\begin{bmatrix}0&\sigma^2&0&-\varsigma\sigma^2&0&0&0&0&0\end{bmatrix}\bigl((\id_{M_9}\tens\jmath_1)\VV\bigr)
\begin{bmatrix}
\jmath_2(e_{-1}e_{-1})\\
\jmath_2(e_{-1}e_{0})\\
\jmath_2(e_{-1}e_{1})\\
\jmath_2(e_{0}e_{-1})\\
\jmath_2(e_{0}e_{0})\\
\jmath_2(e_{0}e_{1})\\
\jmath_2(e_{1}e_{-1})\\
\jmath_2(e_{1}e_{0})\\
\jmath_2(e_{1}e_{1})
\end{bmatrix}
\end{align*}
is the product of the row matrix
\[
\resizebox{\textwidth}{!}{\ensuremath{
\begin{bmatrix}0&\sigma^2\jmath_1(\bv_{-1,-1})&-\varsigma\jmath_1(\bv_{-1,0})&-\varsigma\sigma^2\jmath_1(\bv_{-1,-1})&(1-\varsigma^2)\jmath_1(\bv_{-1,0})&\sigma^2\jmath_1(\bv_{-1,1})&\varsigma\jmath_1(\bv_{-1,0})&-\varsigma\sigma^2\jmath_1(\bv_{-1,1})&0\end{bmatrix}
}}
\]
by the column vector $\left[\begin{smallmatrix}\jmath_2(e_{-1}e_{-1})\\\vdots\\\jmath_2(e_{1}e_{1})\end{smallmatrix}\right]$. This we immediately find to be
\begin{align*}
\jmath_1(\bv_{-1,-1})\jmath_2(\sigma^2e_{-1}e_0-\varsigma\sigma^2e_0e_{-1})
&+\jmath_1(\bv_{-1,0})\jmath_2\bigl(-{\varsigma}e_{-1}e_1+(1-\varsigma^2)e_0^2+{\varsigma}e_1e_{-1}\bigr)\\&\quad+\jmath_1(\bv_{-1,1})\jmath_2(\sigma^2e_0e_1-\varsigma\sigma^2e_1e_0)\\
&=\jmath_1(\bv_{-1,-1})\jmath_2({\lambda}e_{-1})+\jmath_1(\bv_{-1,0})\jmath_2({\lambda}e_0)+\jmath_1(\bv_{-1,1})\jmath_2({\lambda}e_1)\\
&={\lambda}P_{-1},
\end{align*}
which proves \eqref{P-1}.

Relations \eqref{P0} and \eqref{P1} are verified in an analogous manner using the remaining equalities from the proof of Proposition \ref{onPi}. Finally it is clear that ${P_i}^*=P_{-i}$ for all $i$.

Ad \eqref{gestoscXX}. Applying $(\id\btens_\zeta\Gamma_{q,\rho,\lambda})\comp\Gamma_{q,\rho,\lambda}$ to a generator $e_i$ we obtain
{\allowdisplaybreaks
\begin{align*}
(\id\btens_\zeta\Gamma_{q,\rho,\lambda})\Gamma_{q,\rho,\lambda}(e_i)&=(\id\btens_\zeta\Gamma_{q,\rho,\lambda})\biggl(\,\sum_{k=-1}^1\jmath_1(\bv_{i,k})\jmath_2(e_k)\biggr)\\
&=\sum_{k=-1}^1\jmath_1(\bv_{i,k})\biggl(\,\sum_{l=-1}^1\jmath_2(\bv_{k,l})\jmath_3(e_l)\biggr)\\
&=\sum_{l=-1}^1\biggl(\,\sum_{k=-1}^1\jmath_1(\bv_{i,k})\jmath_2(\bv_{k,l}\biggr)\jmath_3(e_l)\\
&=\sum_{l=-1}^1\jmath_1\bigl(\DSUq(\bv_{i,l})\bigr)\jmath_2(e_l)=(\DSUq\btens_\zeta\id)\Gamma_{q,\rho,\lambda}(e_i),
\end{align*}
}where $\jmath_1,\jmath_2,\jmath_3$ in the second and third line are the canonical inclusions of factors in the triple braided product $\C(\SUq)\btens_\zeta\C(\SUq)\btens_\zeta\C(\XX_{q,\rho,\lambda})$ while $\jmath_1$ and $\jmath_2$ in the last line refer to the double product $\bigl(\C(\SUq)\btens_\zeta\C(\SUq)\bigr)\btens_\zeta\C(\XX_{q,\rho,\lambda})$.

As for the density condition we note that since the matrix $\bV\in{M_3(\C(\SUq))}$ is unitary, we have for each $i$
\[
\jmath_2(e_i)=\sum_{k=-1}^1\jmath_1(\bv_{k,i}^*)\Gamma_{q,\rho,\lambda}(e_k),
\]
and hence
\[
\jmath_2(\C(\XX_{q,\rho,\lambda}))\subset\bigl[\jmath_1(\C(\SUq))\Gamma_{q,\rho,\lambda}(\C(\XX_{q,\rho,\lambda}))\bigr].
\]
Also clearly
\[
\jmath_1(\C(\SUq))\subset\bigl[\jmath_1(\C(\SUq))\Gamma_{q,\rho,\lambda}(\C(\XX_{q,\rho,\lambda}))\bigr],
\]
which yields the desired result.

Ad \eqref{PodlesXX}. It can be easily checked by direct computation that the defining relations of $\C(\XX_{q,\rho,\lambda})$ when expressed in the re-scaled generators
\[
E_1=\tfrac{1}{\sigma^2}e_{-1},\quad
E_0=\tfrac{1}{\sqrt{\varsigma}}e_0,\quad
E_{-1}=\tfrac{1}{\sigma^2}e_1
\]
read
\begin{align*}
\sigma^2\bigl(E_{-1}E_1+\tfrac{1}{\varsigma}E_1E_{-1}\bigr)+E_0^2&=\rho'\I,\\
E_1E_0-{\varsigma}E_0E_1&=\lambda'E_1,\\
\sigma^2(E_{-1}E_1-E_1E_{-1})+(1-\varsigma)E_0^2&={\lambda'}E_0,\\
E_0E_{-1}-{\varsigma}E_{-1}E_0&={\lambda'}E_{-1}
\end{align*}
which are precisely relations (2b), (2e), (2d) and (2c) of \cite[Section 3]{spheres} (with $\lambda'=\tfrac{\lambda}{\sqrt{\varsigma}\sigma^2}$ and $\rho'=\tfrac{\rho}{\varsigma\sigma^2}$). Clearly we still have ${E_i}^*=E_{-i}$ for $i=-1,0,1$.
\end{proof}

\begin{definition}
A \emph{braided quantum sphere} is a quantum space $\XX$ such that $\C(\XX)\in\operatorname{Ob}(\CstT)$ and there exists $\Gamma\in\Mor_\TT(\C(\XX),\C(\SUq)\btens_\zeta\C(\XX))$ such that $(\Gamma\btens_\zeta\id)\comp\Gamma=(\id\btens_\zeta\DSUq)\comp\Gamma$ and conditions \eqref{simple1}--\eqref{simple3} of Section \ref{simplespaces} hold.
\end{definition}

In other words quantum spheres for the braided quantum $\SUq$ are quantum spaces with an action (in the braided sense) of $\SUq$ which has the basic properties of the action on the quotient sphere described in Section \ref{theQuotientS2}. Moreover one can interpret conditions \eqref{simple1}--\eqref{simple3} as saying that if $\XX$ is a quantum sphere for the braided quantum $\SUq$ and $\Gamma:\C(\XX)\to\C(\SUq)\btens_\zeta\C(\XX)$ is the action then
\begin{itemize}
\item the action is ergodic or in other words $\XX$ is a (braided quantum) homogeneous space of $\SUq$ (condition \eqref{simple1}),
\item the multiplicity of the three-dimensional irreducible representation in the spectrum of the action is equal to $1$ and the \cst-algebra $\C(\XX)$ is generated by $\I$ and the corresponding spectral subspace (conditions \eqref{simple2} and \eqref{simple2.5}),
\item the action of $\TT$ on $\C(\XX)$ is analogous to the action on $\C(\SS^2_q)$ (cf.~Section \ref{inTerms}).
\end{itemize}

\begin{corollary}
The family braided quantum spheres coincides with the family of quantum spheres defined by Podle\'s in \cite{spheres}. The quantum spheres corresponding to the braided quantum $\operatorname{SU}(2)$ group with complex parameter $q$ are the Podle\'s spheres for the real parameter $|q|$.
\end{corollary}

The family of quantum spheres has several descriptions by generators and relations. Section 2.5 of the article \cite{ldgeometry} contains a very concise yet thorough description of these quantum spaces. The first observation is that by scaling the generators $e_{-1},e_0,e_1$ of $\C(\XX_{q,\rho,\lambda})$ by a constant $\theta\in\RR\mysetminus\{0\}$ we obtain elements satisfying the relations defining $\C(\XX_{q,\theta^2\rho,\theta\lambda})$. This means that the two algebras are isomorphic via an isomorphism $\Phi$, but additionally $\Phi$ is $\SUq$-equivariant in the sense that $(\id\btens_\zeta\Phi)\comp\Gamma_{q,\rho,\lambda}=\Gamma_{q,\theta^2\rho,\theta\lambda}\comp\Phi$. It follows that for a fixed $q$ each element of the family of quantum spaces $\{\XX_{q,\rho,\lambda}\}$ is $\SUq$-equivariantly isomorphic to one of the family of Podle\'{s} spheres
\[
\bigl\{\SS^2_{|q|,c}\bigr\}_{c\in\RR\cup\infty}
\]
originally introduced in \cite[Section 3]{spheres} (for $c<0$ the corresponding \cst-algebras are finite-dimensional or $\{0\}$).

It is also easy to see that if $\C(\XX_{q,\rho,\lambda})$ is $\SUq$-equivariantly isomorphic to $\C(\XX_{q,\rho',\lambda'})$ then there exists $\kappa\in\RR\mysetminus\{0\}$ such that the isomorphism maps the generators $e_{-1},e_0,e_1$ of $\C(\XX_{q,\rho,\lambda})$ to $\kappa$ times the corresponding generators of $\C(\XX_{q,\rho',\lambda'})$. The reason for this is the assumed uniqueness of a subspace $\mathsf{W}'\subset\C(\XX_{q,\rho',\lambda'})$ possessing a basis which transforms according to $\bV$ and the fact that such a basis is unique up to proportionality. The constant must be real because ${e_i}^*=e_{-i}$ for all $i$ (cf.~Section \ref{simplespaces}).

\section*{Acknowledgments}

The author wishes to thank Matthew Daws, Piotr M.~Hajac and Paweł Kasprzak for helpful discussions and comments.

\setcounter{section}{0}

\renewcommand{\thesection}{\Alph{section}}

\section{Appendix}

\subsection{Calculation of certain products of matrix elements of $\bV$}\label{vv}

Below we will express the matrix elements of the $9\times9$ matrix $\VV$ discussed in Remark \ref{bigV} in terms of the linear generating set $\bigl\{a_{n,k,l}\bigr\}_{n\in\ZZ,k,l\in\ZZ_+}$ from Lemma \ref{baza}f.  We only use \eqref{SUq2rel} and \eqref{add}.

{\allowdisplaybreaks
\begin{align*}
\bv_{-1,-1}\bv_{-1,-1}&=\alpha^2\cdot\alpha^2=a_{4,0,0},\\
\overline{\zeta}\bv_{-1,-1}\bv_{-1,0}&=\overline{\zeta}\alpha^2\cdot(-\sigma^2)\gamma^*\alpha=-\overline{\zeta}q^{-1}\sigma^2\alpha^3\gamma^*=-\overline{\zeta}q^{-1}\sigma^2a_{3,0,1},\\
\overline{\zeta}^2\bv_{-1,-1}\bv_{-1,1}&=\overline{\zeta}^2\alpha^2\cdot(-q){\gamma^*}^2=-\overline{\zeta}^2qa_{2,0,2},\\
\bv_{-1,0}\bv_{-1,-1}&=-\sigma^2\gamma^*\alpha\cdot\alpha^2=-q^{-3}\sigma^2\alpha^3\gamma^*=-q^{-3}\sigma^2a_{3,0,1},\\
\bv_{-1,0}\bv_{-1,0}&=-\sigma^2\gamma^*\alpha\cdot(-\sigma^2)\gamma^*\alpha=q^{-3}\sigma^4\alpha^2{\gamma^*}^2=q^{-3}\sigma^4a_{2,0,2},\\
\bv_{-1,0}\bv_{-1,1}&=-\sigma^2\gamma^*\alpha\cdot(-q){\gamma^*}^2=q\sigma^2\gamma^*\alpha{\gamma^*}^2=\sigma^2\alpha{\gamma^*}^3=\sigma^2a_{1,0,3},\\
\bv_{-1,1}\bv_{-1,-1}&=-q{\gamma^*}^2\cdot\alpha^2=-q^{-3}\alpha^2{\gamma^*}^2=-q^{-3}a_{2,0,2},\\
\zeta\bv_{-1,1}\bv_{-1,0}&=\zeta(-q){\gamma^*}^2\cdot(-\sigma^2)\gamma^*\alpha={\zeta}q\sigma^2{\gamma^*}^3\alpha={\zeta}q^{-2}\sigma^2\alpha{\gamma^*}^3=\varsigma^{-1}\sigma^2a_{1,0,3},\\
\zeta^2\bv_{-1,1}\bv_{-1,1}&=\zeta^2(-q){\gamma^*}^2\cdot(-q){\gamma^*}^2=\zeta^2q^2{\gamma^*}^4=\zeta^2q^2a_{0,0,4},\\
\zeta\bv_{-1,-1}\bv_{0,-1}&=\zeta\alpha^2\cdot\zeta\alpha\gamma=\zeta^2\alpha^3\gamma=\zeta^2a_{3,1,0},\\
\bv_{-1,-1}\bv_{0,0}&=\alpha^2\cdot(\I-\sigma^2\gamma^*\gamma)=\alpha^2-\sigma^2\alpha^2\gamma\gamma^*=a_{2,0,0}-\sigma^2a_{2,1,1},\\
\overline{\zeta}\bv_{-1,-1}\bv_{0,1}&=\overline{\zeta}\alpha^2\cdot\gamma^*\alpha=\overline{\zeta}q\alpha(\alpha\alpha^*)\gamma^*=\overline{q}\alpha(\I-\varsigma\gamma^*\gamma)\gamma^*=\overline{q}\alpha\gamma^*-\overline{q}\varsigma\alpha\gamma{\gamma^*}^2\\&=\overline{q}a_{1,0,1}-\overline{q}{\varsigma}a_{1,1,2},\\
\bv_{-1,0}\bv_{0,-1}&=-\sigma^2\gamma^*\alpha\cdot\zeta\alpha\gamma=-{\zeta}q^{-2}\sigma^2\alpha^2\gamma\gamma^*=-\varsigma^{-1}\sigma^2\alpha^2\gamma\gamma^*=-\varsigma^{-1}\sigma^2a_{2,1,1},\\
\bv_{-1,0}\bv_{0,0}&=-\sigma^2\gamma^*\alpha\cdot(\I-\sigma^2\gamma^*\gamma)=-q^{-1}\sigma^2\alpha\gamma^*+q^{-1}\sigma^4\alpha\gamma{\gamma^*}^2\\&=-q^{-1}\sigma^2a_{1,0,1}+q^{-1}\sigma^4a_{1,1,2},\\
\bv_{-1,0}\bv_{0,1}&=-\sigma^2\gamma^*\alpha\cdot\gamma^*\alpha^*=-q\sigma^2\gamma^*\alpha\alpha^*\gamma^*=-q\sigma^2\gamma^*(\I-\varsigma\gamma^*\gamma)\gamma^*\\&=-q\sigma^2{\gamma^*}^2+q\varsigma\sigma^2\gamma{\gamma^*}^3=-q\sigma^2a_{0,0,2}+q\varsigma\sigma^2a_{0,1,3},\\
\overline{\zeta}\bv_{-1,1}\bv_{0,-1}&=\overline{\zeta}(-q){\gamma^*}^2\cdot\zeta\alpha\gamma=-q{\gamma^*}^2\alpha\gamma=-q^{-1}\alpha\gamma{\gamma^*}^2=-q^{-1}a_{1,1,2},\\
\bv_{-1,1}\bv_{0,0}&=-q{\gamma^*}^2\cdot(\I-\sigma^2\gamma^*\gamma)=-q{\gamma^*}^2+q\sigma^2\gamma{\gamma^*}^3=-qa_{0,0,2}+q\sigma^2a_{0,1,3},\\
\zeta\bv_{-1,1}\bv_{0,1}&=\zeta(-q){\gamma^*}^2\cdot\gamma^*\alpha^*=-{\zeta}q^4\alpha^*{\gamma^*}^3=-{\zeta}q^4a_{-1,0,3},\\
\zeta^2\bv_{-1,-1}\bv_{1,-1}&=\zeta^2\alpha^2\cdot(-q)\zeta\gamma^2=-\zeta^3q\alpha^2\gamma^2=-\zeta^3qa_{2,2,0},\\
\zeta\bv_{-1,-1}\bv_{1,0}&=\zeta\alpha^2\cdot(-\sigma^2)\alpha^*\gamma=-{\zeta}\sigma^2\alpha(\alpha\alpha^*)\gamma=-{\zeta}\sigma^2\alpha(\I-\varsigma\gamma^*\gamma)\gamma=-{\zeta}\sigma^2\alpha\gamma+{\zeta}\varsigma\sigma^2\alpha\gamma^2\gamma^*\\&=-{\zeta}\sigma^2a_{1,1,0}+{\zeta}\varsigma\sigma^2a_{1,2,1}=-{\zeta}\sigma^2a_{1,1,0}+q^2\sigma^2a_{1,2,1},\\
\bv_{-1,-1}\bv_{1,1}&=\alpha^2\cdot{\alpha^*}^2=\alpha(\I-\varsigma\gamma^*\gamma)\alpha^*=\alpha\alpha^*-\varsigma\alpha\gamma^*\gamma\alpha^*=\alpha\alpha^*-\varsigma^2\alpha\alpha^*\gamma^*\gamma\\&=(\I-\varsigma\gamma^*\gamma)-\varsigma^2(\I-\varsigma\gamma^*\gamma)\gamma^*\gamma=\I-\varsigma(1+\varsigma)\gamma\gamma^*+\varsigma^3\gamma^2{\gamma^*}^2\\&=\I-\varsigma\sigma^2\gamma\gamma^*+\varsigma^3\gamma^2{\gamma^*}^2=\I-\varsigma\sigma^2a_{0,1,1}+\varsigma^3a_{0,2,2},\\
\bv_{-1,0}\bv_{1,-1}&=-\sigma^2\gamma^*\alpha\cdot(-q)\zeta\gamma^2={\zeta}q\sigma^2\gamma^*\alpha\gamma^2={\zeta}\sigma^2\alpha\gamma^2\gamma^*={\zeta}\sigma^2a_{1,2,1},\\
\bv_{-1,0}\bv_{1,0}&=-\sigma^2\gamma^*\alpha\cdot(-\sigma^2)\alpha^*\gamma=\sigma^4\gamma^*(\I-\varsigma\gamma^*\gamma)\gamma=\sigma^4\gamma\gamma^*-\varsigma\sigma^4\gamma^2{\gamma^*}^2\\&=\sigma^4a_{0,1,1}-\varsigma\sigma^4a_{0,2,2},\\
\bv_{-1,0}\bv_{1,1}&=-\sigma^2\gamma^*\alpha\cdot{\alpha^*}^2=-\sigma^2\gamma^*(\I-\varsigma\gamma^*\gamma)\alpha^*=-\sigma^2\gamma^*\alpha^*+\varsigma\sigma^2\gamma{\gamma^*}^2\alpha^*\\&=-q\sigma^2\alpha^*\gamma^*+q^2\overline{q}\varsigma\sigma^2\alpha^*\gamma{\gamma^*}^2=-q\sigma^2\alpha^*\gamma^*+q\varsigma^2\sigma^2\alpha^*\gamma{\gamma^*}^2\\&=-q\sigma^2a_{-1,0,1}+q\varsigma^2\sigma^2a_{-1,1,2},\\
\overline{\zeta}^2\bv_{-1,1}\bv_{1,-1}&=\overline{\zeta}^2(-q){\gamma^*}^2\cdot(-q)\zeta\gamma^2=\overline{\zeta}q^2\gamma^2{\gamma^*}^2=\varsigma\gamma^2{\gamma^*}^2={\varsigma}a_{0,2,2},\\
\overline{\zeta}\bv_{-1,1}\bv_{1,0}&=\overline{\zeta}(-q){\gamma^*}^2\cdot(-\sigma^2)\alpha^*\gamma=\overline{\zeta}q\sigma^2{\gamma^*}^2\alpha^*\gamma=\overline{\zeta}q^3\sigma^2\alpha^*\gamma{\gamma^*}^2=q\varsigma\sigma^2\alpha^*\gamma{\gamma^*}^2\\&=q\varsigma\sigma^2a_{-1,1,2},\\
\bv_{-1,1}\bv_{1,1}&=-q{\gamma^*}^2\cdot{\alpha^*}^2=-q^5{\alpha^*}^2{\gamma^*}^2=-q^5a_{-2,0,2},\\
\bv_{0,-1}\bv_{-1,-1}&=\zeta\alpha\gamma\cdot\alpha^2=\zeta\overline{q}^{-2}\alpha^2\gamma=\zeta\overline{q}^{-2}a_{3,1,0},\\
\overline{\zeta}\bv_{0,-1}\bv_{-1,0}&=\alpha\gamma\cdot(-\sigma^2)\gamma^*\alpha=-\sigma^2\alpha\gamma\gamma^*\alpha=-\varsigma^{-1}\sigma^2\alpha^2\gamma\gamma^*=-\varsigma^{-1}\sigma^2a_{2,1,1},\\
\overline{\zeta}^2\bv_{0,-1}\bv_{-1,1}&=\overline{\zeta}\alpha\gamma\cdot(-q){\gamma^*}^2=-\overline{\zeta}q\alpha\gamma{\gamma^*}^2=-\overline{q}\alpha\gamma{\gamma^*}^2=-\overline{q}a_{1,1,2},\\
\bv_{0,0}\bv_{-1,-1}&=(\I-\sigma^2\gamma^*\gamma)\cdot\alpha^2=\alpha^2-\sigma^2\gamma^*\gamma\alpha^2=\alpha^2-\varsigma^{-2}\sigma^2\alpha^2\gamma\gamma^*=a_{2,0,0}-\varsigma^{-2}\sigma^2a_{2,1,1},\\
\bv_{0,0}\bv_{-1,0}&=(\I-\sigma^2\gamma^*\gamma)\cdot(-\sigma^2)\gamma^*\alpha=-q^{-1}\sigma^2\alpha\gamma^*+\sigma^4\gamma{\gamma^*}^2\alpha\\&=-q^{-1}\sigma^2\alpha\gamma^*+q^{-2}\overline{q}^{-1}\sigma^4\alpha\gamma{\gamma^*}^2=-q^{-1}\sigma^2\alpha\gamma^*+q^{-1}\varsigma^{-1}\sigma^4\alpha\gamma{\gamma^*}^2\\&=-q^{-1}\sigma^2a_{1,0,1}+q^{-1}\varsigma^{-1}\sigma^4a_{1,1,2},\\
\bv_{0,0}\bv_{-1,1}&=(\I-\sigma^2\gamma^*\gamma)\cdot(-q){\gamma^*}^2=-q{\gamma^*}^2+q\sigma^2\gamma{\gamma^*}^3=-qa_{0,0,2}+q\sigma^2a_{0,1,3},\\
\bv_{0,1}\bv_{-1,-1}&=\gamma^*\alpha^*\cdot\alpha^2=\gamma^*(\I-\gamma^*\gamma)\alpha=\gamma^*\alpha-\gamma{\gamma^*}^2\alpha=q^{-1}\alpha\gamma^*-q^{-2}\overline{q}^{-1}\alpha\gamma{\gamma^*}^2\\&=q^{-1}\alpha\gamma^*-q^{-1}\varsigma^{-1}\alpha\gamma{\gamma^*}^2=q^{-1}a_{1,0,1}-q^{-1}\varsigma^{-1}a_{1,1,2},\\
\zeta\bv_{0,1}\bv_{-1,0}&=\zeta\gamma^*\alpha^*\cdot(-\sigma^2)\gamma^*\alpha=-{\zeta}q^{-1}\sigma^2\gamma^*\alpha^*\alpha\gamma^*=-{\zeta}q^{-1}\sigma^2\gamma^*(\I-\gamma^*\gamma)\gamma^*\\&={\zeta}q^{-1}\sigma^2{\gamma^*}^2+{\zeta}q^{-1}\sigma^2\gamma{\gamma^*}^3=\overline{q}^{-1}\sigma^2{\gamma^*}^2+\overline{q}^{-1}\sigma^2\gamma{\gamma^*}^3\\&=\overline{q}^{-1}\sigma^2a_{0,0,2}+\overline{q}^{-1}\sigma^2a_{0,1,3},\\
\zeta^2\bv_{0,1}\bv_{-1,1}&=\zeta^2\gamma^*\alpha^*\cdot(-q){\gamma^*}^2=-\zeta^2q^2\alpha^*{\gamma^*}^3=-\zeta^2q^2a_{-1,0,3},\\
\zeta\bv_{0,-1}\bv_{0,-1}&=\zeta^2\alpha\gamma\cdot\zeta\alpha\gamma=\zeta^3\alpha\gamma\alpha\gamma=\zeta^3\overline{q}^{-1}\alpha^2\gamma^2=\zeta^3\overline{q}^{-1}a_{2,2,0},\\
\bv_{0,-1}\bv_{0,0}&=\zeta\alpha\gamma\cdot(\I-\sigma^2\gamma^*\gamma)=\zeta\alpha\gamma-{\zeta}\sigma^2\alpha\gamma^2\gamma^*={\zeta}a_{1,1,0}-{\zeta}\sigma^2a_{1,2,1},\\
\overline{\zeta}\bv_{0,-1}\bv_{0,1}&=\overline{\zeta}\zeta\alpha\gamma\cdot\gamma^*\alpha^*=\varsigma\alpha\alpha^*\gamma\gamma^*=\varsigma(\I-\varsigma\gamma^*\gamma)\gamma\gamma^*=\varsigma\gamma\gamma^*-\varsigma^2\gamma^2{\gamma^*}^2\\&={\varsigma}a_{0,1,1}-\varsigma^2a_{0,2,2},\\
\bv_{0,0}\bv_{0,-1}&=(\I-\sigma^2\gamma^*\gamma)\cdot\zeta\alpha\gamma=\zeta\alpha\gamma-{\zeta}\sigma^2\gamma^*\gamma\alpha\gamma=\zeta\alpha\gamma-{\zeta}\varsigma^{-1}\sigma^2\alpha\gamma^2\gamma^*\\&={\zeta}a_{1,1,0}-{\zeta}\varsigma^{-1}\sigma^2a_{1,2,1},\\
\bv_{0,0}\bv_{0,0}&=(\I-\sigma^2\gamma^*\gamma)\cdot(\I-\sigma^2\gamma^*\gamma)=\I-2\sigma^2\gamma\gamma^*+\sigma^4\gamma^2{\gamma^*}^2=\I-2\sigma^2a_{0,1,1}+\sigma^4a_{0,2,2},\\
\bv_{0,0}\bv_{0,1}&=(\I-\sigma^2\gamma^*\gamma)\cdot\gamma^*\alpha^*=\gamma^*\alpha^*-\sigma^2\gamma{\gamma^*}^2\alpha^*=q\alpha^*\gamma^*-q^2\overline{q}\sigma^2\alpha^*\gamma{\gamma^*}^2\\&=q\alpha^*\gamma^*-q\varsigma\sigma^2\alpha^*\gamma{\gamma^*}^2=qa_{-1,0,1}-q\varsigma\sigma^2a_{-1,1,2},\\
\overline{\zeta}\bv_{0,1}\bv_{0,-1}&=\overline{\zeta}\gamma^*\alpha^*\cdot\zeta\alpha\gamma=\gamma^*\alpha^*\alpha\gamma=\gamma^*(\I-\gamma^*\gamma)\gamma=\gamma\gamma^*-\gamma^2{\gamma^*}^2=a_{0,1,1}-a_{0,2,2},\\
\bv_{0,1}\bv_{0,0}&=\gamma^*\alpha^*\cdot(\I-\sigma^2\gamma^*\gamma)=q\alpha^*\gamma^*(\I-\sigma^2\gamma^*\gamma)=q\alpha^*\gamma^*-q\sigma^2\gamma{\gamma^*}^2\\&=qa_{-1,0,1}-q\sigma^2a_{-1,1,2},\\
\zeta\bv_{0,1}\bv_{0,1}&=\zeta\gamma^*\alpha^*\cdot\gamma^*\alpha^*={\zeta}q^3{\alpha^*}^2{\gamma^*}^2={\zeta}q^3a_{-2,0,2},\\
\zeta^2\bv_{0,-1}\bv_{1,-1}&=\zeta^3\alpha\gamma\cdot(-q)\zeta\gamma^2=-\zeta^4q\alpha\gamma^3=-\zeta^4qa_{1,3,0},\\
\zeta\bv_{0,-1}\bv_{1,0}&=\zeta^2\alpha\gamma\cdot(-\sigma^2)\alpha^*\gamma=-\zeta^2\overline{q}\sigma^2\alpha\alpha^*\gamma^2=-{\zeta}q\sigma^2(\I-\varsigma\gamma^*\gamma)\gamma^2\\&=-{\zeta}q\sigma^2\gamma^2-{\zeta}q\varsigma\sigma^2\gamma^3\gamma^*=-{\zeta}q\sigma^2\gamma^2-q^3\sigma^2\gamma^3\gamma^*=-{\zeta}q\sigma^2a_{0,2,0}-q^3\sigma^2a_{0,3,1},\\
\bv_{0,-1}\bv_{1,1}&=\zeta\alpha\gamma\cdot{\alpha^*}^2=\zeta\overline{q}^2\alpha{\alpha^*}^2\gamma=\varsigma(\alpha\alpha^*)\alpha^*\gamma=\varsigma(\I-\varsigma\gamma^*\gamma)\alpha^*\gamma=\varsigma\alpha^*\gamma-\varsigma^2\gamma\gamma^*\alpha^*\gamma\\&=\varsigma\alpha^*\gamma-\varsigma^3\alpha^*\gamma^2\gamma^*={\varsigma}a_{-1,1,0}-\varsigma^3a_{-1,2,1},\\
\bv_{0,0}\bv_{1,-1}&=(\I-\sigma^2\gamma^*\gamma)\cdot(-q)\zeta\gamma^2=-{\zeta}q\gamma^2+{\zeta}q\sigma^2\gamma^3\gamma^*=-{\zeta}qa_{0,2,0}+{\zeta}q\sigma^2a_{0,3,1},\\
\bv_{0,0}\bv_{1,0}&=(\I-\sigma^2\gamma^*\gamma)\cdot(-\sigma^2)\alpha^*\gamma=-\sigma^2\alpha^*\gamma+\sigma^4\gamma^*\gamma\alpha^*\gamma=-\sigma^2\alpha^*\gamma+\varsigma\sigma^4\alpha^*\gamma^2\gamma^*\\&=-\sigma^2a_{-1,1,0}+\varsigma\sigma^4a_{-1,2,1},\\
\bv_{0,0}\bv_{1,1}&=(\I-\sigma^2\gamma^*\gamma)\cdot{\alpha^*}^2={\alpha^*}^2-\sigma^2\gamma^*\gamma{\alpha^*}^2={\alpha^*}^2-\varsigma^2{\alpha^*}^2\sigma^2\gamma\gamma^*\\&=a_{-2,0,0}-\varsigma^2\sigma^2a_{-2,1,1},\\
\overline{\zeta}^2\bv_{0,1}\bv_{1,-1}&=\overline{\zeta}^2\gamma^*\alpha^*\cdot(-q)\zeta\gamma^2=-\overline{\zeta}q\gamma^*\alpha^*\gamma^2=-\overline{\zeta}q^2\alpha^*\gamma^2\gamma^*=-\varsigma\alpha^*\gamma^2\gamma^*=-{\varsigma}a_{-1,2,1},\\
\overline{\zeta}\bv_{0,1}\bv_{1,0}&=\overline{\zeta}\gamma^*\alpha^*\cdot(-\sigma^2)\alpha^*\gamma=-\overline{\zeta}\sigma^2\gamma^*{\alpha^*}^2\gamma=-\overline{\zeta}q^2\sigma^2{\alpha^*}^2\gamma\gamma^*=-\varsigma\sigma^2{\alpha^*}^2\gamma\gamma^*\\&=-\varsigma\sigma^2a_{-2,1,1},\\
\bv_{0,1}\bv_{1,1}&=\gamma^*\alpha^*\cdot{\alpha^*}^2=q^3{\alpha^*}^3\gamma^*=q^3a_{-3,0,1},\\
\bv_{1,-1}\bv_{-1,-1}&=-q\zeta\gamma^2\cdot\alpha^2=-{\zeta}q\overline{q}^{-4}\alpha^2\gamma^2=-\zeta^5q^{-3}a_{2,2,0},\\
\overline{\zeta}\bv_{1,-1}\bv_{-1,0}&=-q\gamma^2\cdot(-\sigma^2)\gamma^*\alpha=q\sigma^2\gamma^2\gamma^*\alpha=\overline{q}^{-2}\sigma^2\alpha\gamma^2\gamma^*=\overline{q}^{-2}\sigma^2a_{1,2,1},\\
\overline{\zeta}^2\bv_{1,-1}\bv_{-1,1}&=-\overline{\zeta}q\gamma^2\cdot(-q){\gamma^*}^2=\overline{\zeta}q^2\gamma^2{\gamma^*}^2=\varsigma\gamma^2{\gamma^*}^2={\varsigma}a_{0,2,2},\\
\bv_{1,0}\bv_{-1,-1}&=-\sigma^2\alpha^*\gamma\cdot\alpha^2
=-\overline{q}^{-2}\sigma^2(\alpha^*\alpha)\alpha\gamma=-\overline{q}^{-2}\sigma^2(\I-\gamma^*\gamma)\alpha\gamma\\&=-\overline{q}^{-2}\sigma^2\alpha\gamma+\overline{q}^{-2}\sigma^2\gamma^*\gamma\alpha\gamma=-\overline{q}^{-2}\sigma^2\alpha\gamma+\overline{q}^{-2}\varsigma^{-1}\sigma^2\alpha\gamma^*\gamma^2\\&=-\overline{q}^{-2}\sigma^2a_{1,1,0}+\overline{q}^{-2}\varsigma^{-1}\sigma^2a_{1,2,1},\\
\bv_{1,0}\bv_{-1,0}&=-\sigma^2\alpha^*\gamma\cdot(-\sigma^2)\gamma^*\alpha
=\sigma^4\alpha^*\gamma\gamma^*\alpha=\varsigma^{-1}\sigma^4\alpha^*\alpha\gamma\gamma^*=\varsigma^{-1}\sigma^4(\I-\gamma^*\gamma)\gamma\gamma^*\\&=\varsigma^{-1}\sigma^4\gamma\gamma^*-\varsigma^{-1}\sigma^4\gamma^2{\gamma^*}^2=\varsigma^{-1}\sigma^4a_{0,1,1}-\varsigma^{-1}\sigma^4a_{0,2,2},\\
\bv_{1,0}\bv_{-1,1}&=-\sigma^2\alpha^*\gamma\cdot(-q){\gamma^*}^2=q\sigma^2\alpha^*\gamma{\gamma^*}^2=q\sigma^2a_{-1,1,2},\\
\bv_{1,1}\bv_{-1,-1}&={\alpha^*}^2\cdot\alpha^2=\alpha^*(\I-\gamma^*\gamma)\alpha=\alpha^*\alpha-\alpha^*\gamma^*\gamma\alpha=\alpha^*\alpha-\varsigma^{-1}\alpha^*\alpha\gamma^*\gamma\\&=(\I-\gamma^*\gamma)(\I-\varsigma^{-1}\gamma^*\gamma)=\I-(1+\varsigma^{-1})\gamma^*\gamma+\varsigma^{-1}\gamma^2{\gamma^*}^2\\&=\I-\varsigma^{-1}\sigma^2\gamma^*\gamma+\varsigma^{-1}\gamma^2{\gamma^*}^2=\I-\varsigma^{-1}\sigma^2a_{0,1,1}+\varsigma^{-1}a_{0,2,2},\\
\zeta\bv_{1,1}\bv_{-1,0}&=\zeta{\alpha^*}^2\cdot(-\sigma^2)\gamma^*\alpha=-{\zeta}\sigma^2{\alpha^*}^2\gamma^*\alpha=-{\zeta}q^{-1}\sigma^2\alpha^*(\alpha^*\alpha)\gamma^*\\&=-\overline{q}^{-1}\sigma^2\alpha^*(\I-\gamma^*\gamma)\gamma^*=-\overline{q}^{-1}\sigma^2\alpha^*\gamma^*+\overline{q}^{-1}\sigma^2\alpha^*\gamma{\gamma^*}^2\\&=-\overline{q}^{-1}\sigma^2a_{-1,0,1}+\overline{q}^{-1}\sigma^2a_{-1,1,2},\\
\zeta^2\bv_{1,1}\bv_{-1,1}&=\zeta^2{\alpha^*}^2\cdot(-q){\gamma^*}^2=-\zeta^2qa_{-2,0,2},\\
\zeta\bv_{1,-1}\bv_{0,-1}&=-\zeta^2q\gamma^2\cdot\zeta\alpha\gamma=-\zeta^3q\gamma^2\alpha\gamma=-\zeta^3q\overline{q}^{-2}\alpha\gamma^3=-\zeta^5q^{-1}\alpha\gamma^3=-\zeta^5q^{-1}a_{1,3,0},\\
\bv_{1,-1}\bv_{0,0}&=-q\zeta\gamma^2\cdot(\I-\sigma^2\gamma^*\gamma)=-{\zeta}q\gamma^2+{\zeta}q\sigma^2\gamma^3\gamma^*=-{\zeta}qa_{0,2,0}+{\zeta}q\sigma^2a_{0,3,1},\\
\overline{\zeta}\bv_{1,-1}\bv_{0,1}&=-q\gamma^2\cdot\gamma^*\alpha^*=-\varsigma^2\alpha^*\gamma^2\gamma^*=-\varsigma^2a_{-1,2,1},\\
\bv_{1,0}\bv_{0,-1}&=-\sigma^2\alpha^*\gamma\cdot\zeta\alpha\gamma=-{\zeta}\sigma^2\alpha^*\gamma\alpha\gamma=-\zeta\overline{q}^{-1}\sigma^2\alpha^*\alpha\gamma^2=-\zeta\overline{q}^{-1}\sigma^2(\I-\gamma^*\gamma)\gamma^2\\&=-\zeta\overline{q}^{-1}\sigma^2\gamma^2+\zeta\overline{q}^{-1}\sigma^2\gamma^3\gamma^*=-\zeta\overline{q}^{-1}\sigma^2a_{0,2,0}+\zeta\overline{q}^{-1}\sigma^2a_{0,3,1},\\
\bv_{1,0}\bv_{0,0}&=-\sigma^2\alpha^*\gamma\cdot(\I-\sigma^2\gamma^*\gamma)=-\sigma^2\alpha^*\gamma+\sigma^4\alpha^*\gamma^2\gamma^*=-\sigma^2a_{-1,1,0}+\sigma^4a_{-1,2,1},\\
\bv_{1,0}\bv_{0,1}&=-\sigma^2\alpha^*\gamma\cdot\gamma^*\alpha^*=-\varsigma\sigma^2{\alpha^*}^2\gamma\gamma^*=-\varsigma\sigma^2a_{-2,1,1},\\
\overline{\zeta}\bv_{1,1}\bv_{0,-1}&=\overline{\zeta}{\alpha^*}^2\cdot\zeta\alpha\gamma=\alpha^*(\alpha^*\alpha)\gamma=\alpha^*(\I-\gamma^*\gamma)\gamma=\alpha^*\gamma-\alpha^*\gamma^2\gamma^*=a_{-1,1,0}-a_{-1,2,1},\\
\bv_{1,1}\bv_{0,0}&={\alpha^*}^2\cdot(\I-\sigma^2\gamma^*\gamma)={\alpha^*}^2-\sigma^2{\alpha^*}^2\gamma\gamma^*=a_{-2,0,0}-\sigma^2a_{-2,1,1},\\
\zeta\bv_{1,1}\bv_{0,1}&=\zeta{\alpha^*}^2\cdot\gamma^*\alpha^*={\zeta}q{\alpha^*}^3\gamma^*={\zeta}qa_{-3,0,1},\\
\zeta^2\bv_{1,-1}\bv_{1,-1}&=-q\zeta^3\gamma^2\cdot(-q)\zeta\gamma^2=\zeta^4q^2\gamma^4=\zeta^4q^2a_{0,4,0},\\
\zeta\bv_{1,-1}\bv_{1,0}&=-q\zeta^2\gamma^2\cdot(-\sigma^2)\alpha^*\gamma=\zeta^2q\sigma^2\gamma^2\alpha^*\gamma=\zeta^2q\overline{q}^2\sigma^2\alpha^*\gamma^3=q^3\sigma^2\alpha^*\gamma^3=q^3\sigma^2a_{-1,3,0},\\
\bv_{1,-1}\bv_{1,1}&=-q\zeta\gamma^2\cdot{\alpha^*}^2=-{\zeta}q\overline{q}^4{\alpha^*}^2\gamma^2=-\overline{q}\varsigma^2{\alpha^*}^2\gamma^2=-\overline{q}\varsigma^2a_{-2,2,0},\\
\bv_{1,0}\bv_{1,-1}&=-\sigma^2\alpha^*\gamma\cdot(-q)\zeta\gamma^2={\zeta}q\sigma^2\alpha^*\gamma^3={\zeta}q\sigma^2a_{-1,3,0},\\
\bv_{1,0}\bv_{1,0}&=-\sigma^2\alpha^*\gamma\cdot(-\sigma^2)\alpha^*\gamma=\sigma^4\alpha^*\gamma\alpha^*\gamma=\overline{q}\sigma^4{\alpha^*}^2\gamma^2=\overline{q}\sigma^4a_{-2,2,0},\\
\bv_{1,0}\bv_{1,1}&=-\sigma^2\alpha^*\gamma\cdot{\alpha^*}^2=-\overline{q}^2\sigma^2{\alpha^*}^3\gamma=-\overline{q}^2\sigma^2a_{-3,1,0},\\
\overline{\zeta}^2\bv_{1,1}\bv_{1,-1}&=\overline{\zeta}^2{\alpha^*}^2\cdot(-q)\zeta\gamma^2=-\overline{\zeta}q{\alpha^*}^2\gamma^2=-\overline{q}{\alpha^*}^2\gamma^2=-\overline{q}a_{-2,2,0},\\
\overline{\zeta}\bv_{1,1}\bv_{1,0}&=\overline{\zeta}{\alpha^*}^2\cdot(-\sigma^2)\alpha^*\gamma=-\overline{\zeta}\sigma^2{\alpha^*}^3\gamma=-\overline{\zeta}\sigma^2a_{-3,1,0},\\
\bv_{1,1}\bv_{1,1}&={\alpha^*}^2\cdot{\alpha^*}^2=a_{-4,0,0}.
\end{align*}
} 

\subsection{Direct check that $\bV$ is a representation}\label{directcheck}

Using \eqref{SUq2rel} supplemented by \eqref{add} and \eqref{braid} we find that:
{\allowdisplaybreaks
\begin{align*}
\DSUq(\alpha^2)&=\bigl(\jmath_1(\alpha)\jmath_2(\alpha)-q\jmath_1(\gamma^*)\jmath_2(\gamma)\bigr)\bigl(\jmath_1(\alpha)\jmath_2(\alpha)-q\jmath_1(\gamma^*)\jmath_2(\gamma)\bigr)\\
&=\jmath_1(\alpha^2)\jmath_2(\alpha^2)-q\jmath_1(\gamma^*\alpha)\jmath_2(\gamma\alpha)-q\jmath_1(\alpha\gamma^*)\jmath_2(\alpha\gamma)+{\zeta}q^2\jmath_1\bigl({\gamma^*}^2\bigr)\jmath_2(\gamma^2)\\
&=\jmath_1(\alpha^2)\jmath_2(\alpha^2)-\zeta\jmath_1(\gamma^*\alpha)\jmath_2(\overline{q}\gamma\alpha)-\overline{q}\jmath_1(\alpha\gamma^*)\jmath_2(\zeta\alpha\gamma)+\jmath_1\bigl({-q\gamma^*}^2\bigr)\jmath_2(-q\zeta\gamma^2)\\
&=\jmath_1(\alpha^2)\jmath_2(\alpha^2)-\jmath_1(\gamma^*\alpha)\jmath_2(\zeta\alpha\gamma)-\overline{q}\jmath_1(q\gamma^*\alpha)\jmath_2(\zeta\alpha\gamma)+\jmath_1\bigl({-q\gamma^*}^2\bigr)\jmath_2(-q\zeta\gamma^2)\\
&=\jmath_1(\alpha^2)\jmath_2(\alpha^2)+\jmath_1(-\sigma^2\gamma^*\alpha)\jmath_2(\zeta\alpha\gamma)+\jmath_1\bigl({-q\gamma^*}^2\bigr)\jmath_2(-q\zeta\gamma^2),\\
\DSUq(\zeta\alpha\gamma)&=\zeta\bigl(\jmath_1(\alpha)\jmath_2(\alpha)-q\jmath_1(\gamma^*)\jmath_2(\gamma)\bigr)\bigl(\jmath_1(\gamma)\jmath_2(\alpha)+\jmath_1(\alpha^*)\jmath_2(\gamma)\bigr)\\
&=\jmath_1(\zeta\alpha\gamma)\jmath_2(\alpha^2)-\zeta\overline{\zeta}q\jmath_1(\gamma^*\gamma)\jmath_2(\gamma\alpha)+\jmath_1(\alpha\alpha^*)\jmath_2(\zeta\alpha\gamma)-{\zeta}q\jmath_1(\gamma^*\alpha^*)\jmath_2(\gamma^2)\\
&=\jmath_1(\zeta\alpha\gamma)\jmath_2(\alpha^2)-\jmath_1(\gamma^*\gamma)\jmath_2(\zeta\alpha\gamma)+\jmath_1(\alpha\alpha^*)\jmath_2(\zeta\alpha\gamma)-{\zeta}q\jmath_1(\gamma^*\alpha^*)\jmath_2(\gamma^2)\\
&=\jmath_1(\zeta\alpha\gamma)\jmath_2(\alpha^2)+\jmath_1(\alpha\alpha^*-\gamma^*\gamma)\jmath_2(\zeta\alpha\gamma)+\jmath_1(\gamma^*\alpha^*)\jmath_2(-q\zeta\gamma^2)\\
&=\jmath_1(\zeta\alpha\gamma)\jmath_2(\alpha^2)+\jmath_1(\I-\sigma^2\gamma^*\gamma)\jmath_2(\zeta\alpha\gamma)+\jmath_1(\gamma^*\alpha^*)\jmath_2(-q\zeta\gamma^2),\\
\DSUq(-q\zeta\gamma^2)&=-q\zeta\bigl(\jmath_1(\gamma)\jmath_2(\alpha)+\jmath_1(\alpha^*)\jmath_2(\gamma)\bigr)\bigl(\jmath_1(\gamma)\jmath_2(\alpha)+\jmath_1(\alpha^*)\jmath_2(\gamma)\bigr)\\
&=\jmath_1(-q\zeta\gamma^2)\jmath_2(\alpha^2)-\zeta\overline{\zeta}q\jmath_1(\alpha^*\gamma)\jmath_2(\gamma\alpha)-q\jmath_1(\gamma\alpha^*)\jmath_2(\zeta\alpha\gamma)\\&\quad+\jmath_1\bigl({\alpha^*}^2\bigr)\jmath_2(-q\zeta\gamma^2)\\
&=\jmath_1(-q\zeta\gamma^2)\jmath_2(\alpha^2)-\jmath_1(\alpha^*\gamma)\jmath_2(\zeta\alpha\gamma)-\jmath_1(\varsigma\gamma\alpha^*)\jmath_2(\zeta\alpha\gamma)+\jmath_1\bigl({\alpha^*}^2\bigr)\jmath_2(-q\zeta\gamma^2)\\
&=\jmath_1(-q\zeta\gamma^2)\jmath_2(\alpha^2)+\jmath_1(-\sigma^2\alpha^*\gamma)\jmath_2(\zeta\alpha\gamma)+\jmath_1\bigl({\alpha^*}^2\bigr)\jmath_2(-q\zeta\gamma^2),\\
\DSUq(-\sigma^2\gamma^*\alpha)&=-\sigma^2\bigl(\jmath_1(\gamma)\jmath_2(\alpha)+\jmath_1(\alpha^*)\jmath_2(\gamma)\bigr)^*\bigl(\jmath_1(\alpha)\jmath_2(\alpha)-q\jmath_1(\gamma^*)\jmath_2(\gamma)\bigr)\\
&=-\sigma^2\bigl(\jmath_2(\alpha^*)\jmath_1(\gamma^*)+\jmath_2(\gamma^*)\jmath_1(\alpha)\bigr)\bigl(\jmath_1(\alpha)\jmath_2(\alpha)-q\jmath_1(\gamma^*)\jmath_2(\gamma)\bigr)\\
&=-\sigma^2\jmath_2(\alpha^*)\jmath_1(\gamma^*\alpha)\jmath_2(\alpha)-\sigma^2\jmath_2(\gamma^*)\jmath_1(\alpha^2)\jmath_2(\alpha)+q\sigma^2\jmath_2(\alpha^*)\jmath_1\bigl({\gamma^*}^2\bigr)\jmath_2(\gamma)\\
&\quad+q\sigma^2\jmath_2(\gamma^*)\jmath_1(\alpha\gamma^*)\jmath_2(\gamma)\\
&=\jmath_1(-\sigma^2\gamma^*\alpha)\jmath_2(\alpha^*\alpha)+\jmath_1(\alpha^2)\jmath_2(-\sigma^2\gamma^*\alpha)+\jmath_1\bigl(-q{\gamma^*}^2\bigr)\jmath_2(-\sigma^2\alpha^*\gamma)\\
&\quad+\overline{\zeta}q\sigma^2\jmath_1(\alpha\gamma^*)\jmath_2(\gamma^*\gamma)\\
&=\jmath_1(-\sigma^2\gamma^*\alpha)\jmath_2(\alpha^*\alpha)+\jmath_1(\alpha^2)\jmath_2(-\sigma^2\gamma^*\alpha)+\jmath_1\bigl(-q{\gamma^*}^2\bigr)\jmath_2(-\sigma^2\alpha^*\gamma)\\
&\quad+\overline{q}\sigma^2\jmath_1(q\gamma^*\alpha)\jmath_2(\gamma^*\gamma)\\
&=\jmath_1(-\sigma^2\gamma^*\alpha)\jmath_2(\alpha^*\alpha)+\jmath_1(\alpha^2)\jmath_2(-\sigma^2\gamma^*\alpha)+\jmath_1\bigl(-q{\gamma^*}^2\bigr)\jmath_2(-\sigma^2\alpha^*\gamma)\\
&\quad-\jmath_1(-\sigma^2\gamma^*\alpha)\jmath_2(\varsigma\gamma^*\gamma)\\
&=\jmath_1(\alpha^2)\jmath_2(-\sigma^2\gamma^*\alpha)+\jmath_1(-\sigma^2\gamma^*\alpha)\jmath_2(\alpha^*\alpha-\varsigma\gamma^*\gamma)+\jmath_1\bigl(-q{\gamma^*}^2\bigr)\jmath_2(-\sigma^2\alpha^*\gamma)\\
&=\jmath_1(\alpha^2)\jmath_2(-\sigma^2\gamma^*\alpha)+\jmath_1(-\sigma^2\gamma^*\alpha)\jmath_2(\I-\sigma^2\gamma^*\gamma)+\jmath_1\bigl(-q{\gamma^*}^2\bigr)\jmath_2(-\sigma^2\alpha^*\gamma),\\
\DSUq(\I-\sigma^2\gamma^*\gamma)&=\jmath_1(\I)\jmath_2(\I)-\sigma^2\bigl(\jmath_2(\alpha^*)\jmath_1(\gamma^*)+\jmath_2(\gamma^*)\jmath_1(\alpha)\bigr)\bigl(\jmath_1(\gamma)\jmath_2(\alpha)+\jmath_1(\alpha^*)\jmath_2(\gamma)\bigr)\\
&=\I-\sigma^2\jmath_2(\alpha^*)\jmath_1(\gamma^*\gamma)\jmath_2(\alpha)-\sigma^2\jmath_2(\gamma^*)\jmath_1(\alpha\gamma)\jmath_2(\alpha)-\sigma^2\jmath_2(\alpha^*)\jmath_1(\gamma^*\alpha^*)\jmath_2(\gamma)\\
&\quad-\sigma^2\jmath_2(\gamma^*)\jmath_1(\alpha\alpha^*)\jmath_2(\gamma)\\
&=\I-\sigma^2\jmath_1(\gamma^*\gamma)\jmath_2(\alpha^*\alpha)-\sigma^2\zeta\jmath_1(\alpha\gamma)\jmath_2(\gamma^*\alpha)-\sigma^2\jmath_1(\gamma^*\alpha^*)\jmath_2(\alpha^*\gamma)\\
&\quad-\sigma^2\jmath_1(\alpha\alpha^*)\jmath_2(\gamma^*\gamma)\\
&=\I-\sigma^2\jmath_1(\gamma^*\gamma)\bigl(\I-\jmath_2(\gamma^*\gamma)\bigr)+\jmath_1(\zeta\alpha\gamma)\jmath_2(-\sigma^2\gamma^*\alpha)+\jmath_1(\gamma^*\alpha^*)\jmath_2(-\sigma^2\alpha^*\gamma)\\
&\quad-\sigma^2\bigl(\I-\varsigma\jmath_1(\gamma^*\gamma)\bigr)\jmath_2(\gamma^*\gamma)\\
&=\I-\sigma^2\jmath_1(\gamma^*\gamma)-\sigma^2\jmath_2(\gamma^*\gamma)+\sigma^2\jmath_1(\gamma^*\gamma)\jmath_2(\gamma^*\gamma)+\sigma^2\varsigma\jmath_1(\gamma^*\gamma)\jmath_2(\gamma^*\gamma)\\
&\quad+\jmath_1(\zeta\alpha\gamma)\jmath_2(-\sigma^2\gamma^*\alpha)+\jmath_1(\gamma^*\alpha^*)\jmath_2(-\sigma^2\alpha^*\gamma)\\
&=\jmath_1(\I-\sigma^2\gamma^*\gamma)\jmath_2(\I-\sigma^2\gamma^*\gamma)+\jmath_1(\zeta\alpha\gamma)\jmath_2(-\sigma^2\gamma^*\alpha)+\jmath_1(\gamma^*\alpha^*)\jmath_2(-\sigma^2\alpha^*\gamma)\\
&=\jmath_1(\zeta\alpha\gamma)\jmath_2(-\sigma^2\gamma^*\alpha)+\jmath_1(\I-\sigma^2\gamma^*\gamma)\jmath_2(\I-\sigma^2\gamma^*\gamma)+\jmath_1(\gamma^*\alpha^*)\jmath_2(-\sigma^2\alpha^*\gamma),\\
\DSUq(-\sigma^2\alpha^*\gamma)&=\DSUq(-\sigma^2\gamma^*\alpha)^*\\
&=\bigl(\jmath_1(\alpha^2)\jmath_2(-\sigma^2\gamma^*\alpha)+\jmath_1(-\sigma^2\gamma^*\alpha)\jmath_2(\I-\sigma^2\gamma^*\gamma)\\&\quad+\jmath_1\bigl(-q{\gamma^*}^2\bigr)\jmath_2(-\sigma^2\alpha^*\gamma)\bigr)^*\\
&=\jmath_2(-\sigma^2\alpha^*\gamma)\jmath_1\bigl({\alpha^*}^2\bigr)+\jmath_2(\I-\sigma^2\gamma^*\gamma)\jmath_1(-\sigma^2\alpha^*\gamma)+\jmath_2(-\sigma^2\gamma^*\alpha)\jmath_1(-\overline{q}\gamma^2)\\
&=\jmath_1\bigl({\alpha^*}^2\bigr)\jmath_2(-\sigma^2\alpha^*\gamma)+\jmath_1(-\sigma^2\alpha^*\gamma)\jmath_2(\I-\sigma^2\gamma^*\gamma)+\zeta^2\jmath_1(-\overline{q}\gamma^2)\jmath_2(-\sigma^2\gamma^*\alpha)\\
&=\jmath_1\bigl({\alpha^*}^2\bigr)\jmath_2(-\sigma^2\alpha^*\gamma)+\jmath_1(-\sigma^2\alpha^*\gamma)\jmath_2(\I-\sigma^2\gamma^*\gamma)+\jmath_1(-q\zeta\gamma^2)\jmath_2(-\sigma^2\gamma^*\alpha)\\
&=\jmath_1(-q\zeta\gamma^2)\jmath_2(-\sigma^2\gamma^*\alpha)+\jmath_1(-\sigma^2\alpha^*\gamma)\jmath_2(\I-\sigma^2\gamma^*\gamma)+\jmath_1\bigl({\alpha^*}^2\bigr)\jmath_2(-\sigma^2\alpha^*\gamma),\\
\DSUq\bigl(-q{\gamma^*}^2\bigr)&=\zeta^2\DSUq(-q\zeta\gamma^2)^*\\
&=\zeta^2\bigl(\jmath_1(-q\zeta\gamma^2)\jmath_2(\alpha^2)+\jmath_1(-\sigma^2\alpha^*\gamma)\jmath_2(\zeta\alpha\gamma)+\jmath_1\bigl({\alpha^*}^2\bigr)\jmath_2(-q\zeta\gamma^2)\bigr)^*\\
&=\zeta^2\bigl(\jmath_2\bigl({\alpha^*}^2\bigr)\jmath_1\bigl(-\overline{q}\overline{\zeta}{\gamma^*}^2\bigr)+\jmath_2(\overline{\zeta}\gamma^*\alpha^*)\jmath_1(-\sigma^2\gamma^*\alpha)+\jmath_2\bigl(-\overline{q}\overline{\zeta}{\gamma^*}^2\bigr)\jmath_1(\alpha^2)\bigr)\\
&=\jmath_2\bigl({\alpha^*}^2\bigr)\jmath_1\bigl(-q{\gamma^*}^2\bigr)+\zeta\jmath_2(\gamma^*\alpha^*)\jmath_1(-\sigma^2\gamma^*\alpha)+\jmath_2\bigl(-q{\gamma^*}^2\bigr)\jmath_1(\alpha^2)\\
&=\jmath_1\bigl(-q{\gamma^*}^2\bigr)\jmath_2\bigl({\alpha^*}^2\bigr)+\zeta\overline{\zeta}\jmath_1(-\sigma^2\gamma^*\alpha)\jmath_2(\gamma^*\alpha^*)+\jmath_1(\alpha^2)\jmath_2\bigl(-q{\gamma^*}^2\bigr)\\
&=\jmath_1(\alpha^2)\jmath_2\bigl(-q{\gamma^*}^2\bigr)+\jmath_1(-\sigma^2\gamma^*\alpha)\jmath_2(\gamma^*\alpha^*)+\jmath_1\bigl(-q{\gamma^*}^2\bigr)\jmath_2\bigl({\alpha^*}^2\bigr),\\
\DSUq(\gamma^*\alpha^*)&=\zeta\DSUq(\zeta\alpha\gamma)^*\\
&=\zeta\bigl(\jmath_1(\zeta\alpha\gamma)\jmath_2(\alpha^2)+\jmath_1(\I-\sigma^2\gamma^*\gamma)\jmath_2(\zeta\alpha\gamma)+\jmath_1(\gamma^*\alpha^*)\jmath_2(-q\zeta\gamma^2)\bigr)^*\\
&=\zeta\bigl(\jmath_2\bigl({\alpha^*}^2\bigr)\jmath_1(\overline{\zeta}\gamma^*\alpha^*)+\jmath_2(\overline{\zeta}\gamma^*\alpha^*)\jmath_1(\I-\sigma^2\gamma^*\gamma)+\jmath_2\bigl(-\overline{q}\overline{\zeta}{\gamma^*}^2\bigr)\jmath_1(\alpha\gamma)\bigr)\\
&=\jmath_2\bigl({\alpha^*}^2\bigr)\jmath_1(\gamma^*\alpha^*)+\jmath_2(\gamma^*\alpha^*)\jmath_1(\I-\sigma^2\gamma^*\gamma)+\jmath_2\bigl(-\overline{q}{\gamma^*}^2\bigr)\jmath_1(\alpha\gamma)\\
&=\jmath_1(\gamma^*\alpha^*)\jmath_2\bigl({\alpha^*}^2\bigr)+\jmath_1(\I-\sigma^2\gamma^*\gamma)\jmath_2(\gamma^*\alpha^*)+\zeta^2\jmath_1(\alpha\gamma)\jmath_2\bigl(-\overline{q}{\gamma^*}^2\bigr)\\
&=\jmath_1(\gamma^*\alpha^*)\jmath_2\bigl({\alpha^*}^2\bigr)+\jmath_1(\I-\sigma^2\gamma^*\gamma)\jmath_2(\gamma^*\alpha^*)+\jmath_1(\zeta\alpha\gamma)\jmath_2\bigl(-q{\gamma^*}^2\bigr)\\
&=\jmath_1(\zeta\alpha\gamma)\jmath_2\bigl(-q{\gamma^*}^2\bigr)+\jmath_1(\I-\sigma^2\gamma^*\gamma)\jmath_2(\gamma^*\alpha^*)+\jmath_1(\gamma^*\alpha^*)\jmath_2\bigl({\alpha^*}^2\bigr),\\
\DSUq\bigl({\alpha^*}^2\bigr)&=\DSUq(\alpha^2)^*\\
&=\bigl(\jmath_1(\alpha^2)\jmath_2(\alpha^2)+\jmath_1(-\sigma^2\gamma^*\alpha)\jmath_2(\zeta\alpha\gamma)+\jmath_1\bigl({-q\gamma^*}^2\bigr)\jmath_2(-q\zeta\gamma^2)\bigr)^*\\
&=\jmath_2\bigl({\alpha^*}^2\bigr)\jmath_1\bigl({\alpha^*}^2\bigr)+\jmath_2(\overline{\zeta}\gamma^*\alpha^*)\jmath_1(-\sigma^2\alpha^*\gamma)+\jmath_2\bigl(-\overline{q}\overline{\zeta}{\gamma^*}^2\bigr)\jmath_1(-\overline{q}\gamma^2)\\
&=\jmath_1\bigl({\alpha^*}^2\bigr)\jmath_2\bigl({\alpha^*}^2\bigr)+\zeta\jmath_1(-\sigma^2\alpha^*\gamma)\jmath_2(\overline{\zeta}\gamma^*\alpha^*)+\zeta^4\jmath_1(-\overline{q}\gamma^2)\jmath_2\bigl(-\overline{q}\overline{\zeta}{\gamma^*}^2\bigr)\\
&=\jmath_1\bigl({\alpha^*}^2\bigr)\jmath_2\bigl({\alpha^*}^2\bigr)+\jmath_1(-\sigma^2\alpha^*\gamma)\jmath_2(\gamma^*\alpha^*)+\zeta^2\jmath_1(-q\zeta\gamma^2)\jmath_2\bigl(-\overline{q}\overline{\zeta}{\gamma^*}^2\bigr)\\
&=\jmath_1\bigl({\alpha^*}^2\bigr)\jmath_2\bigl({\alpha^*}^2\bigr)+\jmath_1(-\sigma^2\alpha^*\gamma)\jmath_2(\gamma^*\alpha^*)+\jmath_1(-q\zeta\gamma^2)\jmath_2\bigl(-q{\gamma^*}^2\bigr)\\
&=\jmath_1(-q\zeta\gamma^2)\jmath_2\bigl(-q{\gamma^*}^2\bigr)+\jmath_1(-\sigma^2\alpha^*\gamma)\jmath_2(\gamma^*\alpha^*)+\jmath_1\bigl({\alpha^*}^2\bigr)\jmath_2\bigl({\alpha^*}^2\bigr).
\end{align*}
} 

\end{document}